\documentclass[11pt]{article}
\usepackage{amssymb}
\usepackage{amsfonts}
\usepackage{amsmath}
\usepackage{geometry}

\newtheorem{theorem}{Theorem}

\newtheorem{corollary}[theorem]{Corollary}

\newtheorem{lemma}[theorem]{Lemma}

\newtheorem{proposition}[theorem]{Proposition}
\newtheorem{remark}[theorem]{Remark}

\newenvironment{proof}[1][Proof]{\noindent\textbf{#1.} }{\ \rule{0.5em}{0.5em}}
\DeclareMathOperator{\var}{Var}
\DeclareMathOperator{\cov}{Cov}
\DeclareMathOperator{\ent}{Ent}

\DeclareMathOperator{\Supp}{Supp}

\DeclareMathOperator{\Hess}{Hess}
\DeclareMathOperator{\ess}{ess}
\DeclareMathOperator{\Int}{Int}
\input{tcilatex}

\begin{document}

\title{Weighted Poincar\'{e} inequalities, concentration inequalities and
tail bounds related to Stein kernels in dimension one}
\author{Adrien Saumard \\
%EndAName
CREST-ENSAI, Universit\'{e} Bretagne Loire}
\maketitle

\begin{abstract}
We investigate links between the so-called Stein's density approach in
dimension one and some functional and concentration inequalities. We show
that measures having a finite first moment and a density with connected
support satisfy a weighted Poincar\'{e} inequality with the weight being the
Stein kernel, that indeed exists and is unique in this case. Furthermore, we
prove weighted log-Sobolev and asymmetric Brascamp-Lieb type inequalities
related to Stein kernels. We also show that existence of a uniformly bounded
Stein kernel is sufficient to ensure a positive Cheeger isoperimetric
constant. Then we derive new concentration inequalities. In particular, we
prove generalized Mills' type inequalities when a Stein kernel is uniformly
bounded and sub-gamma concentration for Lipschitz functions of a variable
with a sub-linear Stein kernel. When some exponential moments are finite, a
general concentration inequality is then expressed in terms of
Legendre-Fenchel transform of the Laplace transform of the Stein kernel.
Along the way, we prove a general lemma for bounding the Laplace transform
of a random variable, that should be useful in many other contexts when
deriving concentration inequalities. Finally, we provide density and tail
formulas as well as tail bounds, generalizing previous results that where
obtained in the context of Malliavin calculus.

\noindent
\end{abstract}

\section{Introduction}

Since its introduction by Charles Stein (%
%TCIMACRO{%
%\TeXButton{\cite{MR0402873, MR882007}}{\cite{MR0402873, MR882007}}}%
%BeginExpansion
\cite{MR0402873, MR882007}%
%EndExpansion
), the so-called Stein's method is a corpus of techniques that revealed
itself very successful in studying probability approximation and convergence
in law (see for instance 
%TCIMACRO{%
%\TeXButton{\cite{MR2732624, MR3727600, MR3595350}}{\cite{MR2732624, MR3727600, MR3595350}} }%
%BeginExpansion
\cite{MR2732624, MR3727600, MR3595350}
%EndExpansion
and references therein). Much less is known regarding the interplay between
Stein's method and functional inequalities. Recently, a series of papers (%
%TCIMACRO{%
%\TeXButton{\cite{LedoucNourdinPeccati:15,  MR3665794,  MR3682667, courtade2017existence}}{\cite{LedoucNourdinPeccati:15,  MR3665794,  MR3682667, courtade2017existence}}}%
%BeginExpansion
\cite{LedoucNourdinPeccati:15,  MR3665794,  MR3682667, courtade2017existence}%
%EndExpansion
) started to fill this gap.

More precisely, Ledoux et al. \cite{LedoucNourdinPeccati:15} provide some
improvement of the log-Sobolev inequality and Talagrand's quadratic
transportation cost inequality through the use of a Stein kernel and in
particular, the Stein discrepancy that measures the closeness of the Stein
kernels to identity. In a second paper \cite{MR3665794}, these authors also
provide a lower bound of the deficit in the Gaussian log-Sobolev inequality
in terms Stein's characterization of the Gaussian distribution. Recently,
Fathi and Nelson \cite{MR3682667} also consider free Stein kernel and use it
to improve the free log-Sobolev inequality. Finally, Courtade et al. \cite%
{courtade2017existence} proved that the existence a reversed weighted Poincar%
\'{e} inequality is sufficient to ensure existence of a Stein kernel. To do
so, they use an elegant argument based on the Lax-Milgram theorem. They also
provide bounds on the Stein discrepancy and application to a quantitative
central limit theorem.

Particularizing to dimension one, the present paper aims at pursuing
investigations about the relations between Stein's method - especially Stein
kernels - and some functional inequalities, together with some concentration
inequalities. The limitation to dimension one comes, for most of the
results, from the one-dimensional nature of the covariance identities given
in Section $2$, that are instrumental for the rest of the paper and which
crucially rely on the use of the properties of the cumulative distribution
function.

We prove that a measure $\nu $ having a finite first moment and a density
with connected support satisfies a weighted Poincar\'{e} inequality in the
sense of \cite{MR2510011}, with the weight being the Stein kernel $\tau
_{\nu }$ (see the definition in Section \ref{section_cov_identity}\ below),
that is unique in this case. More precisely, for any $f\in L_{2}\left( \nu
\right) $, absolutely continuous, we have%
\begin{equation}
%TCIMACRO{\TeXButton{Var}{\var}}%
%BeginExpansion
\var%
%EndExpansion
\left( f\left( X\right) \right) \leq \mathbb{E}\left[ \tau _{\nu }\left(
X\right) \left( f^{\prime }\left( X\right) \right) ^{2}\right] \text{ .}
\label{intro_ineq_weight_Poin}
\end{equation}

The latter inequality allows us to recover by different techniques some
weighted Poincar\'{e} inequalities previously established in \cite%
{BobkovLedoux:14}\ for the Beta distribution or in \cite{MR3519678}\ for the
generalized Cauchy distribution and to highlight new ones, considering for
instance Pearson's class of distributions.

It is also well-known that Muckenhoupt-type criteria characterize (weighted)
Poincar\'{e} and log-Sobolev inequalities on the real line (%
%TCIMACRO{%
%\TeXButton{\cite{MR1845806, MR1682772}}{\cite{MR1845806, MR1682772}}}%
%BeginExpansion
\cite{MR1845806, MR1682772}%
%EndExpansion
). We indeed recover, up to a multiplicative constant, inequality (\ref%
{intro_ineq_weight_Poin}) from the classical Muckenhoupt criterion.
Furthermore, using the criterion first established by Bobkov and G\"{o}tze 
\cite{MR1682772} to characterize log-Sobolev inequalities on the real line,
we prove that, under the conditions ensuring the weighted Poincar\'{e}
inequality (\ref{intro_ineq_weight_Poin}), together with some asymptotic
assumptions on the behavior of the Stein kernel around the edges of the
support of the measure $\nu $, the following inequality holds,%
\begin{equation}
%TCIMACRO{\TeXButton{Ent}{\ent}}%
%BeginExpansion
\ent%
%EndExpansion
_{\nu }\left( g^{2}\right) \leq C_{\nu }\int \tau _{\nu }^{2}\left(
g^{\prime }\right) ^{2}d\nu \text{ ,}  \label{intro_ineq_weight_log_Sob}
\end{equation}%
for some constant $C_{\nu }>0$ and with $%
%TCIMACRO{\TeXButton{Ent}{\ent}}%
%BeginExpansion
\ent%
%EndExpansion
_{\nu }\left( g^{2}\right) =\int g^{2}\log g^{2}d\nu -\int g^{2}d\nu \log
\int g^{2}d\nu $. More precisely, if the support of $\nu $ is compact, then
inequality (\ref{intro_ineq_weight_log_Sob}) is valid if $\tau _{\nu }^{-1}$
is integrable at the edges of the support. If on contrary, an edge of the
support is infinite, then a necessary condition for inequality (\ref%
{intro_ineq_weight_log_Sob}) to hold is that the Stein kernel does not tend
to zero around this edge.

We also derive asymmetric Brascamp-Lieb type inequalities related to the
Stein kernel and show that existence of a uniformly bounded Stein kernel is
sufficient to ensure a positive Cheeger isoperimetric constant.

There is also a growing literature, initiated by Chatterjee (\cite{MR2288072}%
), about the links between Stein's method and concentration inequalities.
Several approaches are considered, from the method of exchangeable pairs (%
%TCIMACRO{%
%\TeXButton{\cite{MR2683635, MR3189061,  MR3551202}}{\cite{MR2288072,MR2683635, MR3189061,  MR3551202}}}%
%BeginExpansion
\cite{MR2288072,MR2683635, MR3189061,  MR3551202}%
%EndExpansion
), to the density approach coupled with Malliavin calculus (%
%TCIMACRO{%
%\TeXButton{\cite{MR2556018, MR2568291, MR3380095, MR3352331}}{\cite{MR2556018, MR2568291, MR3380095, MR3352331}}}%
%BeginExpansion
\cite{MR2556018, MR2568291, MR3380095, MR3352331}%
%EndExpansion
), size biased coupling (%
%TCIMACRO{%
%\TeXButton{\cite{MR2773032,MR2763529,MR2832913,MR3758727}}{\cite{MR2773032,MR2763529,MR2832913,MR3758727}}}%
%BeginExpansion
\cite{MR2773032,MR2763529,MR2832913,MR3758727}%
%EndExpansion
), zero bias coupling (\cite{MR3162712}) or more general Stein couplings (%
\cite{barbour2018central}). As emphasized for instance in the survey by
Chatterjee \cite{MR3727600}, one major strength of Stein-type methodologies
applied to concentration of measure is that it often allows to deal with
dependent and complex system of random variables, finding for instance
applications in statistical mechanics or in random graph theory.

In the present work, we investigate relationships between Stein kernels and
concentration of measure by building upon ideas and exporting techniques
about the use of covariance identities for Gaussian concentration from
Bobkov, G\"{o}tze and Houdr\'{e} \cite{MR1836739}.

Considering first the case where a Stein kernel is uniformly bounded, we
recover the well-known fact that the associated random variable admits a
sub-Gaussian behavior. But we also prove in this setting some refined
concentration inequalities, that we call generalized Mills' type
inequalities, in reference to the classical Mills' inequality for the normal
distribution (see for instance \cite{duembgen2010bounding}). Assume also
that a Stein kernel $\tau _{v}$ exists for the measure $\nu $, is uniformly
bounded, and denote $c=\left\Vert \tau _{v}\right\Vert _{\infty }^{-1}$.
Then the Furthermore, the function $T_{g}\left( r\right) =e^{cr^{2}/2}%
\mathbb{E}\left( g-\mathbb{E}g\right) \mathbf{1}_{\left\{ g-\mathbb{E}g\geq
r\right\} }$ is non-increasing in $r\geq 0$. In particular, for all $r>0$,%
\begin{equation}
\mathbb{P}\left( g-\mathbb{E}g\geq r\right) \leq \mathbb{E}\left( g-\mathbb{E%
}g\right) _{+}\frac{e^{-cr^{2}/2}}{r}\text{ .}  \label{intro_concen_droite}
\end{equation}

In particular, Beta distributions have a bounded Stein kernel and our
concentration inequalities improve on previously best known concentration
inequalities for Beta distributions, recently due to Bobkov and Ledoux \cite%
{BobkovLedoux:14}.

Furthermore, we consider the situation where a Stein kernel has a sub-linear
behavior, recovering and extending in this case sub-Gamma concentration
previously established by Nourdin and Viens \cite{MR2556018}. We also prove
some generalized Mills' type inequalities in this case. More generally, we
prove that the Laplace transform of a Stein kernel controls the Laplace
transform of a Lipschitz function taken on the related distribution. Take $f$
a $1$-Lipschitz function with mean zero with respect to $\nu $ and assume
that $f$ has an exponential moment with respect to $\nu $, that is there
exists $a>0$ such that $\mathbb{E}\left[ e^{af\left( X\right) }\right]
<+\infty $. Then for any $\lambda \in \left( 0,a\right) $,%
\begin{equation}
\mathbb{E}\left[ e^{\lambda f\left( X\right) }\right] \leq \mathbb{E}\left[
e^{\lambda ^{2}\tau _{\nu }\left( X\right) }\right] \text{ .}
\label{intro_Laplace_ineq}
\end{equation}%
It is worth noting that to prove such a result, we state a generic lemma -
Lemma \ref{Lemma_upper_Laplace}\ Section \ref{section_concentration} -
allowing to bound the Laplace transform of a random variable. We believe
that this lemma has an interest by itself, as it may be very convenient when
dealing with Chernoff's method in general.

We also obtain lower tail bounds without the need of Malliavin calculus,
thus extending previous results due to Nourdin and Viens \cite{MR2556018}
and Viens \cite{MR2568291}.

The paper is organized as follows. In Section \ref{section_cov_identity} we
introduce\ some background material, by discussing some well-known and new
formulas for Stein kernels and Stein factors in connection with Menz and
Otto's covariance identity. We also provide formulas involving the Stein
operator, for densities and tails. Then we prove in Section \ref%
{section_weighted_func_ineq} some (weighted) functional inequalities linked
to the behavior of the Stein kernel. In Section \ref{section_concentration}\
we make use of some covariance inequalities to derive various concentration
inequalities for Lipschitz functions of a random variable having a Stein
kernel. Finally, we prove some tail bounds related to the behavior of the
Stein kernel - assumed to be unique - in Section \ref%
{section_density_formula_tail_bounds}.

\section{On covariance identities and the Stein kernel\label%
{section_cov_identity}}

Take a real random variable $X$ of distribution $\nu $ with density $p$ with
respect to the Lebesgue measure on $\mathbb{R}$ and cumulative distribution
function $F$. Assume that the mean of the distribution $\nu $ exists and
denote it by $\mu =\mathbb{E}\left[ X\right] $. Denote also $%
%TCIMACRO{\TeXButton{Supp}{\Supp}}%
%BeginExpansion
\Supp%
%EndExpansion
\left( \nu \right) =\overline{\left\{ x\in \mathbb{R}:p\left( x\right)
>0\right\} }\subset \mathbb{\bar{R}(}:=\mathbb{R\cup }\left\{ -\infty
,+\infty \right\} )$ the support of the measure $\nu $, defined as the
closure of the set where the density is positive and assume that this
support is connected. We denote by $a\in \mathbb{R\cup }\left\{ -\infty
\right\} $, $b\in \mathbb{R}\cup \left\{ +\infty \right\} $, $a<b$, the
edges of $%
%TCIMACRO{\TeXButton{Supp}{\Supp}}%
%BeginExpansion
\Supp%
%EndExpansion
\left( \nu \right) $. For convenience, we also denote by $I\left( \nu
\right) =%
%TCIMACRO{\TeXButton{\Int}{\Int}}%
%BeginExpansion
\Int%
%EndExpansion
\left( 
%TCIMACRO{\TeXButton{Supp}{\Supp}}%
%BeginExpansion
\Supp%
%EndExpansion
\left( \nu \right) \right) $ the interior of the support of $\nu $. The
distribution $\nu $ is said to have a Stein kernel $\tau _{\nu }$, if the
following identity holds true,

\begin{equation}
\mathbb{E}\left[ \left( X-\mu \right) \varphi \left( X\right) \right] =%
\mathbb{E}\left[ \tau _{\nu }\left( X\right) \varphi ^{\prime }\left(
X\right) \right] \text{ ,}  \label{stein_kernel_equation}
\end{equation}%
with $\varphi $ being any differentiable test function such that the
functions $x\mapsto \left( x-\mu \right) \varphi \left( x\right) $ and $%
x\mapsto \tau _{\nu }\left( x\right) \varphi ^{\prime }\left( x\right) $ are 
$\nu $-integrable and $\left[ \tau _{\nu }p\varphi \right] _{a}^{b}=0$. It
is well-known (\cite{LedoucNourdinPeccati:15,courtade2017existence,MR3595350}%
), that under our assumptions the Stein kernel $\tau _{\nu }$ exists, is
unique up to sets of $\nu $-measure zero and a version of the latter is
given by the following formula,%
\begin{equation}
\tau _{v}\left( x\right) =\frac{1}{p\left( x\right) }\int_{x}^{\infty
}\left( y-\mu \right) p\left( y\right) dy\text{ ,}  \label{kernel_formula}
\end{equation}%
for any $x\in I\left( \nu \right) $. Formula (\ref{kernel_formula}) comes
from a simple integration by parts. Notice that $\tau _{\nu }$ is almost
surely positive on the interior of the support of $\nu $.

Although we will focus only on dimension one, it is worth noting that the
definition of a Stein kernel extends to higher dimension, where it is
matrix-valued.\ The question of existence of the Stein kernel for a
particular multi-dimensional measure $\nu $ is nontrivial and only a few
general results are known related to this problem (see for instance \cite%
{LedoucNourdinPeccati:15} \cite{courtade2017existence} and \cite%
{fathi2018stein}). In particular, \cite{courtade2017existence} proves that
the existence of a Stein kernel is ensured whenever a (converse weighted)
Poincar\'{e} inequality is satisfied for the probability measure $\nu $.
Recently, Stein kernels that are positive definite matrices have been
constructed in \cite{fathi2018stein} using transportation techniques.

In this section, that essentially aims at stating some background results
that will be instrumental for the rest of the paper, we will among other
things recover Identity (\ref{kernel_formula}) and introduce a new formula
for the Stein kernel by means of a covariance identity recently obtained in 
\cite{Menz-Otto:2013}\ and further developed in \cite{saumwellner2017efron}.
It actually appears that Menz and Otto's covariance identity is a
consequence of an old result by Hoeffding (see the discussion in \cite%
{Saumard-Wellner:17}).

We define a non-negative and symmetric kernel $k_{\nu }$ on $\mathbb{R}^{2}$
by 
\begin{equation}
k_{\nu }\left( x,y\right) =F\left( x\wedge y\right) -F\left( x\right)
F\left( y\right) ,\qquad \mbox{for all}\left( x,y\right) \in \mathbb{R}^{2}.
\label{def_kernel}
\end{equation}%
For any $p\in \left[ 1,+\infty \right] $, we denote by $L_{p}\left( \nu
\right) $ the space of measurable functions $f$ such that $\left\Vert
f\right\Vert _{p}^{p}=\int \left\vert f\right\vert ^{p}d\nu <+\infty $ for $%
p\in \left[ 1,+\infty \right) $ and $\left\Vert f\right\Vert _{\infty }=%
%TCIMACRO{\TeXButton{ess}{\ess}}%
%BeginExpansion
\ess%
%EndExpansion
\sup_{x\in \mathbb{R}}\left\vert f\left( x\right) \right\vert <+\infty $ for 
$p=+\infty $. If $f\in L_{p}\left( \nu \right) $, $g\in L_{q}\left( \nu
\right) $, $p^{-1}+q^{-1}=1$, we also write%
\begin{equation*}
%TCIMACRO{\TeXButton{Cov}{\cov}}%
%BeginExpansion
\cov%
%EndExpansion
\left( f,g\right) =\int \left( f-\int fd\nu \right) gd\nu 
\end{equation*}%
the covariance of $f$ and $g$ with respect to $\nu $. For $f\in L_{2}\left(
\nu \right) $, we write $%
%TCIMACRO{\TeXButton{Var}{\var}}%
%BeginExpansion
\var%
%EndExpansion
\left( f\right) =%
%TCIMACRO{\TeXButton{Cov}{\cov}}%
%BeginExpansion
\cov%
%EndExpansion
\left( f,f\right) $ the variance of $f$ with respect to $\nu $. For a random
variable $X$ of distribution $\nu $, we will also write $\mathbb{E}\left[
h\left( X\right) \right] =\mathbb{E}\left[ h\right] =\int hd\nu $.

\begin{proposition}[Corollary 2.2, \protect\cite{saumwellner2017efron}]
\label{theorem_cov_id_kernel} If $g$ and $h$ are absolutely continuous and $%
g\in L_{p}(\nu )$, $h\in L_{q}(\nu )$ for some $p\in \lbrack 1,\infty ]$ and 
$p^{-1}+q^{-1}=1$, then 
\begin{equation}
\cov(g,h)=\int \!\!\!\int_{\mathbb{R}^{2}}g^{\prime }(x)k_{\nu
}(x,y)h^{\prime }(y)dxdy\text{.}  \label{cov_id_kernel}
\end{equation}
\end{proposition}

It is worth mentioning that the covariance identity (\ref{cov_id_kernel})
heavily relies on dimension one, since it uses the properties of the
cumulative distribution function $F$ through the kernel $k_{\nu }$. In
dimension greater than or equal to $2$, a covariance identity of the form of
(\ref{cov_id_kernel}) - with derivatives replaced by gradients -, would
actually imply that the measure $\nu $ is Gaussian (for more details, see 
\cite{MR1836739} and also Remark \ref{remark_learning} below).

\begin{remark}
\label{remark_learning}In the context of goodness-of-fit tests, Liu et al. 
\cite{liu2016kernelized} introduce the notion of kernelized Stein
discrepancy as follows. If $K\left( x,y\right) $ is a kernel on $\mathbb{R}%
^{2}$, $p$ and $q$ are two densities and $\left( X,Y\right) $ is a pair of
independent random variables distributed according to $p$, then the
kernelized Stein discrepancy $\mathbb{S}_{K}\left( p,q\right) $ between $p$
and $q$ related to $K$ is%
\begin{equation*}
\mathbb{S}_{K}\left( p,q\right) =\mathbb{E}\left[ \delta _{q,p}\left(
X\right) K\left( X,Y\right) \delta _{q,p}\left( Y\right) \right] \text{ ,}
\end{equation*}%
where $\delta _{q,p}\left( x\right) =\left( \log q\left( x\right) \right)
^{\prime }-\left( \log p\left( x\right) \right) ^{\prime }$ is the
difference between scores of $p$ and $q$. This notion is in fact presented
in \cite{liu2016kernelized} in higher dimension and is used as an efficient
tool to assess the proximity of the laws $p$ and $q$. From formula (\ref%
{cov_id_kernel}), we see that if we take $K_{\nu }\left( x,y\right) =k_{\nu
}(x,y)p_{\nu }\left( x\right) ^{-1}p_{\nu }\left( y\right) ^{-1}$, then we
get the following formula, valid in dimension one,%
\begin{equation*}
\mathbb{S}_{K_{\nu }}\left( p,q\right) =%
%TCIMACRO{\TeXButton{Var}{\var}}%
%BeginExpansion
\var%
%EndExpansion
_{\nu }\left( \log \left( \frac{p}{q}\right) \right) \text{ .}
\end{equation*}%
In higher dimension, Bobkov et al. \cite{MR1836739} proved that the Gaussian
measures satisfy a covariance identity of the same form as in (\ref%
{cov_id_kernel}) above, with derivatives replaced by gradients. More
precisely, let $\left( X,Y\right) $ be a pair of independent normalized
Gaussian vectors in $\mathbb{R}^{d}$, let $\mu _{\alpha }$ be the measure of
the pair $\left( X,\alpha X+\sqrt{1-\alpha ^{2}}Y\right) $ and let $%
p_{N}\left( x,y\right) $ be the density associated the measure $%
\int_{0}^{1}\mu _{\alpha }d\alpha $. Then we have%
\begin{equation*}
%TCIMACRO{\TeXButton{Cov}{\cov}}%
%BeginExpansion
\cov%
%EndExpansion
\left( g\left( X\right) ,h\left( X\right) \right) =\int \!\!\!\int_{\mathbb{R%
}^{2}}\nabla g(x)^{T}p_{N}(x,y)\nabla h(y)dxdy\text{ .}
\end{equation*}%
This gives that for a kernel $K_{N}\left( x,y\right) =p_{N}\left( x,y\right)
\varphi ^{-1}\left( x\right) \varphi ^{-1}\left( y\right) $, where $\varphi $
is the standard normal density on $\mathbb{R}^{d}$, we also have%
\begin{equation*}
\mathbb{S}_{K_{N}}\left( p,q\right) =%
%TCIMACRO{\TeXButton{Var}{\var}}%
%BeginExpansion
\var%
%EndExpansion
\left( \log \left( \frac{p}{q}\right) \left( X\right) \right) \text{ .}
\end{equation*}
\end{remark}

The following formulas will also be useful. They can be seen as special
instances of the previous covariance representation formula.

\begin{corollary}[Corollary 2.1, \protect\cite{saumwellner2017efron}]
\label{corollary_formula_kernel}For an absolutely continuous function $h\in
L_{1}(F)$, 
\begin{equation}
F(z)\int_{\mathbb{R}}hd\nu -\int_{-\infty }^{z}hd\nu =\int_{\mathbb{R}%
}k_{\nu }\left( z,y\right) h^{\prime }(y)dy  \label{rep_g_prime_L1-MRLa}
\end{equation}%
and 
\begin{equation}
-(1-F(z))\int_{\mathbb{R}}hd\nu +\int_{(z,\infty )}hd\nu =\int_{\mathbb{R}%
}k_{\nu }\left( z,y\right) h^{\prime }(y)dy.  \label{rep_g_prime_L1-MRLb}
\end{equation}
\end{corollary}

By combining Theorem \ref{theorem_cov_id_kernel}\ and Corollary \ref%
{corollary_formula_kernel}, we get the following covariance identity.

\begin{proposition}
\label{prop_cov_id}Let $\nu $ be a probability measure on $\mathbb{R}$ and $%
p,q\geq 1$ such that $p^{-1}+q^{-1}=1$. Denote $\mathcal{L}h\left( x\right)
=\int_{x}^{\infty }hd\nu -\left( 1-F\left( x\right) \right) \int_{\mathbb{R}%
}hd\nu =F\left( x\right) \int_{\mathbb{R}}hd\nu -\int_{-\infty }^{x}hd\nu $
for every $x\mathbb{\in \mathbb{R}}$. If $g\in L_{p}\left( \nu \right) $ and 
$h\in L_{q}\left( \nu \right) $ are absolutely continuous and if $g^{\prime }%
\mathcal{L}h$ is integrable with respect to the Lebesgue measure, then%
\begin{equation}
%TCIMACRO{\TeXButton{Cov}{\cov}}%
%BeginExpansion
\cov%
%EndExpansion
\left( g,h\right) =\int_{\mathbb{R}}g^{\prime }\left( x\right) \mathcal{L}%
h\left( x\right) dx\text{ ,}  \label{cov_id_2_gene}
\end{equation}%
Furthermore, if $\nu $ has a density $p$ with respect to the Lebesgue
measure that has a connected support, then%
\begin{equation}
%TCIMACRO{\TeXButton{Cov}{\cov}}%
%BeginExpansion
\cov%
%EndExpansion
\left( g,h\right) =\int_{\mathbb{R}}g^{\prime }\left( x\right) \mathcal{\bar{%
L}}h\left( x\right) p\left( x\right) dx=\mathbb{E}\left[ g^{\prime }\left(
X\right) \mathcal{\bar{L}}h\left( X\right) \right] \text{ ,}
\label{cov_id_connected}
\end{equation}%
where, for every $x\in I\left( \nu \right) $,%
\begin{eqnarray}
\mathcal{\bar{L}}h\left( x\right) &=&p\left( x\right) ^{-1}\mathcal{L}%
h\left( x\right)  \notag \\
&=&\frac{1}{p\left( x\right) }\int_{x}^{\infty }hd\nu -\frac{1-F\left(
x\right) }{p\left( x\right) }\mathbb{E}\left[ h\right] \text{ .}
\label{def_calL_bar}
\end{eqnarray}%
If $x\notin I\left( v\right) $, we take $\mathcal{\bar{L}}h\left( x\right)
=0 $.
\end{proposition}

\begin{proof}
Identity (\ref{cov_id_2_gene}) consists in applying Fubini theorem in the
formula of Theorem \ref{theorem_cov_id_kernel}\ and then using Corollary \ref%
{corollary_formula_kernel}. If $\nu $ has a density $p$ with respect to the
Lebesgue measure that has a connected support then for every $x\notin
I\left( \nu \right) $ we have $\mathcal{L}h\left( x\right) =0$.
Consequently, from Identity (\ref{cov_id_2_gene}) we get, for $g\in
L_{\infty }\left( \nu \right) $, $h\in L_{1}\left( \nu \right) $ absolutely
continuous, 
\begin{eqnarray*}
%TCIMACRO{\TeXButton{Cov}{\cov}}%
%BeginExpansion
\cov%
%EndExpansion
\left( g,h\right) &=&\int_{I\left( \nu \right) }g^{\prime }\left( x\right) 
\mathcal{L}h\left( x\right) dx \\
&=&\int_{I\left( \nu \right) }g^{\prime }\left( x\right) \mathcal{\bar{L}}%
h\left( x\right) p\left( x\right) dx
\end{eqnarray*}%
and so Identity (\ref{cov_id_connected}) is proved.
\end{proof}

From Proposition \ref{prop_cov_id}, we can directly recover formula (\ref%
{kernel_formula}) for the Stein kernel, when it is assumed that the measure
has a connected support and finite first moment. Indeed, by taking $h\left(
x\right) =x-\mu $, we have $h$ $\nu $-integrable and differentiable and so,
for any absolutely continuous function $g\in L_{\infty }\left( \nu \right) $
such that $g^{\prime }\mathcal{\bar{L}}h$ is $\nu $-integrable, applying
Identity (\ref{cov_id_connected}) - since $g^{\prime }\mathcal{L}h=g^{\prime
}\mathcal{\bar{L}}hp$ $\ a.s.$ is Lebesgue integrable - yields 
\begin{equation}
%TCIMACRO{\TeXButton{Cov}{\cov}}%
%BeginExpansion
\cov%
%EndExpansion
\left( g,h\right) =\int_{\mathbb{R}}\left( x-\mu \right) g\left( x\right)
p\left( x\right) dx=\int_{\mathbb{R}}g^{\prime }\mathcal{\bar{L}}hd\nu \text{
.}  \label{cov_id_Lbar}
\end{equation}%
As by a standard approximation argument - i.e. truncations and an
application of the dominated convergence theorem -, identity (\ref%
{cov_id_Lbar}) can be extended to any $g$ such that the functions $x\mapsto
\left( x-\mu \right) g\left( x\right) $ and $x\mapsto \tau _{\nu }\left(
x\right) g^{\prime }\left( x\right) $ are $\nu $-integrable and $\left[ \tau
_{\nu }pg\right] _{a}^{b}=0$, we deduce that a version of the Stein kernel $%
\tau _{v}$ is given by $\mathcal{\bar{L}}h$, which is nothing but the
right-hand side of Identity (\ref{kernel_formula}).

Following the nice recent survey \cite{MR3595350} related to the Stein
method in dimension one, identity (\ref{cov_id_connected}) is exactly the
so-called "generalized Stein covariance identity", written in terms of the
inverse of the Stein operator rather than the Stein operator itself. Indeed,
it is easy to see that the inverse $\mathcal{T}_{\nu }$ of the operator $%
\mathcal{\bar{L}}$ acting on integrable functions with mean zero is given by
the following formula%
\begin{equation*}
\mathcal{T}_{\nu }f=\frac{(fp)^{\prime }}{p}\mathbf{1}_{I\left( \nu \right) }%
\text{ ,}
\end{equation*}%
which is exactly the Stein operator (see Definition 2.1 of \cite{MR3595350}).

It is also well known, see again \cite{MR3595350}, that the inverse of the
Stein operator, that is $\mathcal{\bar{L}}$, is highly involved in deriving
bounds for distances between distributions. From Corollary \ref%
{corollary_formula_kernel}, we have the following seemingly new formula for
this important quantity,%
\begin{equation}
\mathcal{T}_{\nu }^{-1}h\left( x\right) =\mathcal{\bar{L}}h\left( x\right) =%
\frac{1}{p\left( x\right) }\int_{\mathbb{R}}k_{\nu }\left( x,y\right)
h^{\prime }(y)dy\text{ .}  \label{identity_stein_factor}
\end{equation}%
Particularizing the latter identity with $h\left( x\right) =x-\mu $, we
obtain the following identity for the Stein kernel,%
\begin{equation}
\tau _{\nu }\left( x\right) =\frac{1}{p\left( x\right) }\int_{\mathbb{R}%
}k_{\nu }\left( x,y\right) dy\text{ .}  \label{formula_kernel}
\end{equation}%
A consequence of (\ref{formula_kernel}) that will be important in Section %
\ref{section_weighted_func_ineq}\ when deriving weighted functional
inequalities is that for any $x\in I\left( \nu \right) $ the function $%
y\mapsto k_{\nu }\left( x,y\right) (p\left( x\right) \tau _{\nu }\left(
x\right) )^{-1}$ can be seen as the density - with respect to the Lebesgue
measure - of a probability measure, since it is nonnegative and integrates
to one.

We also deduce from (\ref{identity_stein_factor}) the following upper bound,%
\begin{equation*}
\left\vert \mathcal{T}_{\nu }^{-1}h\left( x\right) \right\vert \leq \frac{%
\left\Vert h^{\prime }\right\Vert _{\infty }}{p\left( x\right) }\int_{%
\mathbb{R}}k_{\nu }\left( x,y\right) dy=\frac{\left\Vert h^{\prime
}\right\Vert _{\infty }}{p\left( x\right) }\left( F\left( x\right) \int_{%
\mathbb{R}}xd\nu \left( x\right) -\int_{-\infty }^{x}xd\nu \left( x\right)
\right) \text{ ,}
\end{equation*}%
which is exactly the formula given in Proposition 3.13(a) of \cite{MR3418541}%
.

Let us note $\varphi \left( x\right) =-\log p\left( x\right) $ when $p\left(
x\right) >0$ and $+\infty $ otherwise, the so-called potential of the
density $p$. If on $I\left( \nu \right) $, $\varphi $ has derivative $%
\varphi ^{\prime }\in L_{1}\left( \nu \right) $ absolutely continuous, then
Corollary 2.3 in \cite{saumwellner2017efron} gives 
\begin{equation*}
\int_{\mathbb{R}}k_{\nu }\left( x,y\right) \varphi ^{\prime \prime
}(y)dy=p\left( x\right) \text{ .}
\end{equation*}%
Using the latter identity together with (\ref{identity_stein_factor}), we
deduce the following upper-bound: if $p$ is strictly log-concave (that is $%
\varphi ^{\prime \prime }>0$ on $I\left( \nu \right) $), then%
\begin{equation}
\sup_{x\in I\left( \nu \right) }\left\vert \mathcal{T}_{\nu }^{-1}h\left(
x\right) \right\vert \leq \sup_{x\in I\left( \nu \right) }\frac{\left\vert
h^{\prime }\left( x\right) \right\vert }{\varphi ^{\prime \prime }\left(
x\right) }\text{ .}  \label{ineq_stein_factor_log_con_pot}
\end{equation}%
In particular, if $p$ is $c$-strongly log-concave, meaning that $\varphi
^{\prime \prime }\geq c>0$ on $\mathbb{R}$, then the Stein kernel is
uniformly bounded and $\left\Vert \tau _{\nu }\right\Vert _{\infty }\leq
c^{-1}$. For more about the Stein method related to (strongly) log-concave
measures, see for instance \cite{MR3548768}.

Furthermore, by differentiating (\ref{identity_stein_factor}), we obtain for
any $x\in I\left( \nu \right) $,%
\begin{eqnarray*}
\left( \mathcal{T}_{\nu }^{-1}h\right) ^{\prime }\left( x\right) &=&\varphi
^{\prime }\left( x\right) \mathcal{T}_{\nu }^{-1}h\left( x\right) -h\left(
x\right) -\int_{\mathbb{R}}F\left( y\right) h^{\prime }\left( y\right) dy \\
&=&\varphi ^{\prime }\left( x\right) \mathcal{T}_{\nu }^{-1}h\left( x\right)
-h\left( x\right) +\mathbb{E}\left[ h\left( X\right) \right] \text{ ,}
\end{eqnarray*}%
that is%
\begin{equation*}
\left( \mathcal{T}_{\nu }^{-1}h\right) ^{\prime }\left( x\right) -\varphi
^{\prime }\left( x\right) \mathcal{T}_{\nu }^{-1}h\left( x\right) =-h\left(
x\right) +\mathbb{E}\left[ h\left( X\right) \right] \text{ .}
\end{equation*}%
This is nothing but the so-called Stein equation associated to the Stein
operator.

We conclude this section with the following formulas, that are available
when considering a density with connected support and that will be useful in
the rest of the paper (see in particular Sections \ref{ssection_Muck} and %
\ref{section_density_formula_tail_bounds}).

\begin{proposition}
\label{theorem_formulas}Assume that $X$ is a random variable with
distribution $\nu $ having a density $p$ with connected support with respect
to the Lebesgue measure on $\mathbb{R}$. Take $h\in L_{1}\left( \nu \right) $
with $\mathbb{E}\left[ h\left( X\right) \right] =0$ and assume that the
function $\mathcal{\bar{L}}h$ defined in (\ref{def_calL_bar}) is $\nu $%
-almost surely strictly positive. We have, for any $x_{0},x\in I\left( \nu
\right) $,%
\begin{equation}
p\left( x\right) =\frac{\mathbb{E}\left[ h\left( X\right) 1_{\left\{ X\geq
x_{0}\right\} }\right] }{\mathcal{\bar{L}}h\left( x\right) }\exp \left(
-\int_{x_{0}}^{x}\frac{h\left( y\right) }{\mathcal{\bar{L}}h\left( y\right) }%
dy\right) \text{ .}  \label{formula_density}
\end{equation}%
Consequently, if $X$ has a finite first moment, for any $x\in I\left( \nu
\right) $,%
\begin{equation}
p\left( x\right) =\frac{\mathbb{E}\left[ \left\vert X-\mu \right\vert \right]
}{2\mathcal{\tau }_{\nu }\left( x\right) }\exp \left( -\int_{\mu }^{x}\frac{%
y-\mu }{\mathcal{\tau }_{\nu }\left( y\right) }dy\right) \text{ .}
\label{formula_density_2}
\end{equation}%
By setting $T_{h}\left( x\right) =\exp \left( -\int_{x_{0}}^{x}\frac{h\left(
y\right) }{\mathcal{\bar{L}}h\left( y\right) }dy\right) $ and $I\left( \nu
\right) =\left( a,b\right) $, if the function $h$ is $\nu $-almost positive,
differentiable on $\left( x,b\right) $ and if the ratio $T_{h}\left(
y\right) /h\left( y\right) $ tends to zero when $y$ tends to $b^{-}$, then
we have, for any $x_{0},x\in I\left( \nu \right) $,%
\begin{equation}
\mathbb{P}\left( X\geq x\right) =\mathbb{E}\left[ h\left( X\right)
1_{\left\{ X\geq x_{0}\right\} }\right] \left( \frac{T_{h}\left( x\right) }{%
h\left( x\right) }-\int_{x}^{b}\frac{h^{\prime }\left( y\right) }{%
h^{2}\left( y\right) }T_{h}\left( y\right) dy\right) \text{ .}
\label{formula_tail}
\end{equation}
\end{proposition}

Formula (\ref{formula_density}) can also be found in \cite{MR3418541},
Equation (3.11), under the assumption that $h$ is decreasing and for a
special choice of $x_{0}$. Since $\mathbb{E}\left[ h\left( X\right) \right]
=0$, it is easily seen through its definition (\ref{def_calL_bar}), that if $%
h\neq 0$ $\nu -a.s.$ then $\mathcal{\bar{L}}h>0$ $\nu -a.s.$ When $h=Id-\mu $%
, formulas (\ref{formula_density_2}) and (\ref{formula_tail}) were first
proved respectively\ in \cite{MR2556018} and \cite{MR2568291}, although with
assumption that the random variable $X$ belongs to the space $\mathbf{D}%
^{1,2}$ of square integrable random variables with the natural Hilbert norm
of their Malliavin derivative also square integrable.

In order to take advantage of formulas (\ref{formula_density}) and (\ref%
{formula_tail}), one has to use some information about $\mathcal{\bar{L}}h$.
The most common choice is $h=Id-\mu $, which corresponds to the Stein kernel 
$\mathcal{\bar{L}}\left( Id-\mu \right) =\tau _{\nu }$.

\begin{proof}
Begin with Identity (\ref{formula_density}). As $x_{0}\in I\left( \nu
\right) $ and the function $\mathcal{L}h$ defined in (\ref{prop_cov_id}) is $%
\nu $-almost surely positive, we have for any $x\in I\left( v\right) $,%
\begin{equation*}
\mathcal{L}h\left( x\right) =\mathcal{L}h\left( x_{0}\right) \exp \left(
\int_{x_{0}}^{x}\left( \ln \left( \mathcal{L}h\right) \right) ^{\prime
}\left( y\right) dy\right) \text{ .}
\end{equation*}%
To conclude, note that $\mathcal{L}h\left( x_{0}\right) =\mathbb{E}\left[
h\left( X\right) 1_{\left\{ X\geq x_{0}\right\} }\right] \,$and $\left( \ln
\left( \mathcal{L}h\right) \right) ^{\prime }=-h/\mathcal{\bar{L}}h$. To
prove (\ref{formula_density_2}), simply remark that is follows from (\ref%
{formula_density}) by taking $h=Id-\mu $ and $x_{0}=\mu $.

It remains to prove (\ref{formula_tail}). We have from (\ref{formula_density}%
), $p=\mathbb{E}\left[ h\left( X\right) 1_{\left\{ X\geq x_{0}\right\} }%
\right] T_{h}/\mathcal{\bar{L}}h$ and by definition of $T_{h}$, $%
T_{h}^{\prime }=-hT_{h}/\mathcal{\bar{L}}h$. Hence, integrating between $x$
and $b$ gives%
\begin{eqnarray*}
\mathbb{P}\left( X\geq x\right) &=&\mathbb{E}\left[ h\left( X\right)
1_{\left\{ X\geq x_{0}\right\} }\right] \int_{x}^{b}\frac{T_{h}\left(
y\right) }{\mathcal{\bar{L}}h\left( y\right) }dy \\
&=&\mathbb{E}\left[ h\left( X\right) 1_{\left\{ X\geq x_{0}\right\} }\right]
\int_{x}^{b}\frac{-T_{h}^{\prime }\left( y\right) }{h\left( y\right) }dy \\
&=&\mathbb{E}\left[ h\left( X\right) 1_{\left\{ X\geq x_{0}\right\} }\right]
\left( \left[ \frac{-T_{h}}{h}\right] _{x}^{b}-\int_{x}^{b}\frac{h^{\prime
}\left( y\right) T_{h}\left( y\right) }{h^{2}\left( y\right) }dy\right) \\
&=&\mathbb{E}\left[ h\left( X\right) 1_{\left\{ X\geq x_{0}\right\} }\right]
\left( \frac{T_{h}\left( x\right) }{h\left( x\right) }-\int_{x}^{b}\frac{%
h^{\prime }\left( y\right) T_{h}\left( y\right) }{h^{2}\left( y\right) }%
dy\right) \text{ .}
\end{eqnarray*}
\end{proof}

\section{Some weighted functional inequalities\label%
{section_weighted_func_ineq}}

Weighted functional inequalities appear naturally when generalizing Gaussian
functional inequalities such as Poincar\'{e} and log-Sobolev inequalities.
They were put to emphasis for the generalized Cauchy distribution and more
general $\kappa $-concave distributions by Bobkov and Ledoux 
%TCIMACRO{%
%\TeXButton{\cite{MR2510011,MR2797936}}{\cite{MR2510011,MR2797936}}}%
%BeginExpansion
\cite{MR2510011,MR2797936}%
%EndExpansion
, also in connection with isoperimetric-type problems, weighted Cheeger-type
inequalities and concentration of measure. Then several authors proved
related weighted functional inequalities (%
%TCIMACRO{%
%\TeXButton{\cite{MR2381160,MR3464047,MR3269712,MR3519678,MR2609591,MR3556769,MR2682264}}{\cite{MR2381160,MR3269712,MR3519678,MR3464047,MR2609591,MR3008255,MR3556769,MR2682264}}}%
%BeginExpansion
\cite{MR2381160,MR3269712,MR3519678,MR3464047,MR2609591,MR3008255,MR3556769,MR2682264}%
%EndExpansion
). In the following, we show the strong connection between Stein kernels and
the existence of weighted functional inequalities. Note that a remarkable
first result in this direction was recently established by Courtade et al. 
\cite{courtade2017existence}\ who proved that a reversed weighted Poincar%
\'{e} inequality is sufficient to ensure the existence of a Stein kernel in $%
\mathbb{R}^{d}$, $d\geq 1$.

Our results are derived in dimension one. Indeed, the proofs of weighted
Poincar\'{e} inequalities (Section \ref{ssection_weighted_Poin}) and
asymmetric Brascamp-Lieb type inequalities (Section \ref{ssection_asym_BL})
rely on covariance identities stated in Section \ref{section_cov_identity},
that are based on properties of cumulative distribution functions available
in dimension one. Recently, Fathi \cite{fathi2018stein} generalized in
higher dimension the weighted Poincar\'{e} inequalities derived in Section %
\ref{ssection_weighted_Poin} using transportation techniques. We also derive
in Section \ref{ssection_Muck} some weighted log-Sobolev inequalities, that
are derived through the use of some Muckenhoupt-type criteria that are only
valid in dimension one. However, it is natural to conjecture a
multi-dimensional generalization, especially using tools developed in \cite%
{fathi2018stein}, but this remains an open question. Finally, using a
formula for the isoperimetric constant in dimension one due to Bobkov and
Houdr\'{e} \cite{BobHoud97}, we prove that a uniformly bounded Stein kernel
is essentially sufficient to ensure a positive isoperimetric constant.
Again, the problem in higher dimension seems much more involved and is left
as an interesting open question.

\subsection{Weighted Poincar\'{e}-type inequality\label%
{ssection_weighted_Poin}}

According to \cite{MR2510011}, a measure $\nu $ on $\mathbb{R}$ is said to
satisfy a weighted Poincar\'{e} inequality if there exists a nonnegative
measurable weight function $\omega $ such that for any smooth function $f\in
L_{2}\left( \nu \right) $, 
\begin{equation}
%TCIMACRO{\TeXButton{Var}{\var}}%
%BeginExpansion
\var%
%EndExpansion
\left( f\left( X\right) \right) \leq \mathbb{E}\left[ \omega \left( X\right)
\left( f^{\prime }\left( X\right) \right) ^{2}\right] \text{ .}
\label{def_weight_Poin}
\end{equation}%
The following theorem shows that a probability measure having a finite first
moment and density with connected support on the real line satisfies a
weighted Poincar\'{e} inequality, with the weight being its Stein kernel.

\begin{theorem}
\label{theorem_Poincare_kernel}Take a real random variable $X$ of
distribution $\nu $ with density $p$ with respect to the Lebesgue measure on 
$\mathbb{R}$. Assume that $\mathbb{E}\left[ \left\vert X\right\vert \right]
<+\infty $, $p$ has a connected support and denote $\tau _{\nu }$ the Stein
kernel of $\nu $. Take $f\in L_{2}\left( \nu \right) $, absolutely
continuous. Then%
\begin{equation}
%TCIMACRO{\TeXButton{Var}{\var}}%
%BeginExpansion
\var%
%EndExpansion
\left( f\left( X\right) \right) \leq \mathbb{E}\left[ \tau _{\nu }\left(
X\right) \left( f^{\prime }\left( X\right) \right) ^{2}\right] \text{ .}
\label{ineq_weighted_Poincare}
\end{equation}%
The preceding inequality is optimal whenever $\nu $ admits a finite second
moment, that is $\mathbb{E}\left[ X^{2}\right] <+\infty $, since equality is
reached for $f=Id$, by definition of the Stein kernel.
\end{theorem}

\begin{proof}
We have%
\begin{eqnarray*}
%TCIMACRO{\TeXButton{Var}{\var}}%
%BeginExpansion
\var%
%EndExpansion
\left( f\left( X\right) \right) &=&\mathbb{E}\left[ \sqrt{\tau _{\nu }\left(
X\right) }f^{\prime }\left( X\right) \frac{\mathcal{\bar{L}}f\left( X\right) 
}{\sqrt{\tau _{\nu }\left( X\right) }}\right] \\
&\leq &\sqrt{\mathbb{E}\left[ \tau _{\nu }\left( X\right) \left( f^{\prime
}\left( X\right) \right) ^{2}\right] }\sqrt{\mathbb{E}\left[ \tau _{\nu
}\left( X\right) \left( \frac{\mathcal{\bar{L}}f\left( X\right) }{\tau _{\nu
}\left( X\right) }\right) ^{2}\right] }\text{ .}
\end{eqnarray*}%
By the use of Jensen's inequality, for any $x\in I\left( \nu \right) $,%
\begin{eqnarray*}
\left( \frac{\mathcal{\bar{L}}f\left( x\right) }{\tau _{\nu }\left( x\right) 
}\right) ^{2} &=&\left( \int f^{\prime }\left( y\right) \frac{k_{\nu }\left(
x,y\right) }{\int k_{\nu }\left( x,z\right) dz}dy\right) ^{2} \\
&\leq &\int \left( f^{\prime }\left( y\right) \right) ^{2}\frac{k_{\nu
}\left( x,y\right) }{\int_{z}k_{\nu }\left( x,z\right) dz}dy \\
&=&\int \frac{k_{\nu }\left( x,y\right) }{\tau _{\nu }\left( x\right)
p\left( x\right) }\left( f^{\prime }\left( y\right) \right) ^{2}dy\text{ .}
\end{eqnarray*}%
Hence,%
\begin{eqnarray*}
\mathbb{E}\left[ \tau _{\nu }\left( X\right) \left( \frac{\mathcal{\bar{L}}%
f\left( X\right) }{\tau _{\nu }\left( X\right) }\right) ^{2}\right] &\leq
&\int \tau _{v}\left( x\right) p\left( x\right) \left( \int \frac{k_{\nu
}\left( x,y\right) }{\tau _{\nu }\left( x\right) p\left( x\right) }\left(
f^{\prime }\left( y\right) \right) ^{2}dy\right) dx \\
&=&\int \int k_{\nu }\left( x,y\right) \left( f^{\prime }\left( y\right)
\right) ^{2}dxdy \\
&=&\int \tau _{\nu }\left( y\right) \left( f^{\prime }\left( y\right)
\right) ^{2}p\left( y\right) dy\text{ ,}
\end{eqnarray*}%
which concludes the proof.
\end{proof}

It is worth mentioning that the famous Brascamp-Lieb inequality provides
another weighted Poincar\'{e} inequality in dimension one: if $\nu $ is
strictly log-concave, of density $p=\exp \left( -\varphi \right) $ with a
smooth potential $\varphi $, then for any smooth function $f\in L_{2}\left(
\nu \right) $,%
\begin{equation}
%TCIMACRO{\TeXButton{Var}{\var}}%
%BeginExpansion
\var%
%EndExpansion
\left( f\left( X\right) \right) \leq \mathbb{E}\left[ \left( \varphi
^{\prime \prime }\left( X\right) \right) ^{-1}\left( f^{\prime }\left(
X\right) \right) ^{2}\right] \text{ .}  \label{ineq_BL}
\end{equation}%
In particular, if $\nu $ is strongly log-concave (that is $\varphi ^{\prime
\prime }\geq c>0$ for some constant $c>0$), then both the Brascamp-Lieb
inequality (\ref{ineq_BL}) and inequality (\ref{ineq_weighted_Poincare}) -
combined with the estimate $\left\Vert \tau _{\nu }\right\Vert _{\infty
}\leq c^{-1}$ coming from inequality (\ref{ineq_stein_factor_log_con_pot}) -
imply the Poincar\'{e} inequality $%
%TCIMACRO{\TeXButton{Var}{\var}}%
%BeginExpansion
\var%
%EndExpansion
\left( f\left( X\right) \right) \leq c^{-1}\mathbb{E}\left[ \left( f^{\prime
}\left( X\right) \right) ^{2}\right] $, that also follows from the Bakry-%
\'{E}mery criterion. However, in general, the Stein kernel appearing in (\ref%
{ineq_weighted_Poincare}) may behave differently from the inverse of the
second derivative of the potential and there is no general ordering between
the right-hand sides of inequalities (\ref{ineq_weighted_Poincare}) and (\ref%
{ineq_BL}).

Indeed, let us discuss the situation for a classical class of examples in
functional inequalities, namely the class of Subbotin densities $p_{\alpha
}\left( x\right) =Z_{\alpha }^{-1}\exp \left( -\left\vert x\right\vert
^{\alpha }/\alpha \right) $ for $x\in \mathbb{R}$, where $Z_{\alpha }>0$ is
the normalizing constant. Recall that densities $p_{\alpha }$ are not
strongly log-concave and do not satisfy the Bakry-\'{E}mery criterion -
except for $\alpha =2$ which corresponds to the normal density - but they
satisfy a Poincar\'{e} inequality if and only if $\alpha \geq 1$ and a
log-Sobolev inequality if and only if $\alpha \geq 2$ (see \cite{MR1796718}
and also \cite{MR3269712} for a thorough discussion on optimal constants in
these inequalities). We restrict our discussion to the condition $\alpha >1$
for which $p_{\alpha }$ is strictly log-concave, so that the Brascamp-Lieb
inequality applies. More precisely, by writing $p_{\alpha }=\exp \left(
-\varphi _{\alpha }\right) $, we have $\left( \varphi ^{\prime \prime
}\right) ^{-1}\left( x\right) =\left( \alpha -1\right) ^{-1}\left\vert
x\right\vert ^{2-\alpha }$. Using the explicit formula (\ref{formula_kernel}%
) for the Stein kernel and an integration by parts, one can easily check
that if $\alpha \in \left( 1,2\right) $, then $\tau _{\alpha }\left(
x\right) <\left\vert x\right\vert ^{2-\alpha }$ where $\tau _{\alpha }$ is
the Stein kernel associated to $p_{\alpha }$. We thus get, for $\alpha \in
\left( 1,2\right) $, $\left( \varphi ^{\prime \prime }\right) ^{-1}\left(
x\right) >\tau _{\alpha }\left( x\right) $ for any $x\in \mathbb{R}$ and so
the Brascamp-Lieb inequality is less accurate than Theorem \ref%
{theorem_Poincare_kernel} in this case. Furthermore, if $\alpha >2$ then $%
\tau _{\alpha }\left( x\right) >\left\vert x\right\vert ^{2-\alpha }$ and so 
$\left( \varphi ^{\prime \prime }\right) ^{-1}\left( x\right) <\tau _{\alpha
}\left( x\right) $ for any $x\in \mathbb{R}$, which means that the
Brascamp-Lieb inequality is more accurate than Theorem \ref%
{theorem_Poincare_kernel} for $\alpha >2$. However, we can not recover
through the use of Theorem \ref{theorem_Poincare_kernel} - or the
Brascamp-Lieb inequality - the existence of a spectral gap for $\alpha \in %
\left[ 1,2\right) $. Finally, Theorem \ref{theorem_Poincare_kernel} gives us
the existence of some weighted Poincar\'{e} inequalities for any $\alpha >0$%
, whereas Brascamp-Lieb inequality only applies for $\alpha >1$.

It is also worth mentioning that Theorem \ref{theorem_Poincare_kernel} has
been recently generalized to higher dimension by Fathi \cite{fathi2018stein}%
, for a Stein kernel that is a positive definite matrix and that is defined
through the use of a so-called moment map. The proof of this (non-trivial)
extension of our result is actually based on the Brascamp-Lieb inequality
itself, but applied to a measure also defined through the moment map problem.

Let us now detail some classical examples falling into the setting of
Theorem \ref{theorem_Poincare_kernel}.

The beta distribution $B_{\alpha ,\beta },$ $\alpha ,\beta >0$ is supported
on $\left( 0,1\right) $, with density $p_{\alpha ,\beta }$ given by 
\begin{equation}
p_{\alpha ,\beta }\left( x\right) =\frac{x^{\alpha -1}\left( 1-x\right)
^{\beta -1}}{B\left( \alpha ,\beta \right) }\text{, \ \ }0<x<1\text{.}
\label{def_beta}
\end{equation}%
The normalizing constant $B\left( \alpha ,\beta \right) $ is the classical
beta function of two variables. The beta distribution has been for instance
recently studied in \cite{BobkovLedoux:14} in connection with the analysis
of the rates of convergence of the empirical measure on $\mathbb{R}$ for
some Kantorovich transport distances. The Stein kernel $\tau _{\alpha ,\beta
}$ associated to the Beta distribution is given by $\tau _{\alpha ,\beta
}\left( x\right) =\left( \alpha +\beta \right) ^{-1}x\left( 1-x\right) $ for 
$x\in \left( 0,1\right) $ (see for instance \cite{MR3595350}) and thus
Theorem \ref{theorem_Poincare_kernel} allows to exactly recover Proposition
B.5 of \cite{BobkovLedoux:14} (which is optimal for linear functions as
noticed in \cite{BobkovLedoux:14}). Our techniques are noticeably different
since the weighted Poincar\'{e} inequality is proved in \cite%
{BobkovLedoux:14} by using orthogonal (Jacobi) polynomials. Notice also that
the beta density $p_{\alpha ,\beta }$ is strictly log-concave on the
interior of its support. Indeed, by writing $p_{\alpha ,\beta }=\exp \left(
-\varphi _{\alpha ,\beta }\right) $, we get, for any $x\in \left( 0,1\right) 
$,%
\begin{equation*}
\left( \varphi ^{\prime \prime }\right) ^{-1}\left( x\right) =\frac{%
x^{2}\left( 1-x\right) ^{2}}{\left( \alpha -1\right) \left( 1-x\right)
^{2}+\left( \beta -1\right) x^{2}}\text{ .}
\end{equation*}%
Remark that, for instance, $\left( \varphi ^{\prime \prime }\right)
^{-1}\left( x\right) \sim _{x\rightarrow 0^{+}}x^{2}/\left( \alpha -1\right) 
$, whereas $\tau _{\alpha ,\beta }\left( x\right) \sim _{x\rightarrow
0^{+}}x/\left( \alpha +\beta \right) $, so the weights in the Brascamp-Lieb
inequality (\ref{ineq_BL}) and in inequality (\ref{ineq_weighted_Poincare})
are not of the same order at $0$ (or by symmetry at $1$), although it is
also easy to show that there exists a constant $c_{\alpha ,\beta }>0$ - $%
c_{\alpha ,\beta }=(\alpha +\beta )/\left( 2\sqrt{\left( \alpha -1\right)
\left( \beta -1\right) }\right) $ works - such that $\left( \varphi ^{\prime
\prime }\right) ^{-1}\left( x\right) \leq c_{\alpha ,\beta }\tau _{\alpha
,\beta }\left( x\right) $ for any $x\in \left( 0,1\right) $.

Note that considering Laguerre polynomials, that are eigenfunctions of the
Laguerre operator for which the Gamma distribution is invariant and
reversible, one can also show an optimal weighted Poincar\'{e} inequality
for the Gamma distribution, which include as a special instance the
exponential distribution (see \cite{MR1440138}\ and also \cite%
{bakry2014analysis}, Section 2.7). Theorem \ref{theorem_Poincare_kernel}
also gives an optimal weighted Poincar\'{e} inequality for the Gamma
distribution and more generally for Pearson's class of distributions (see
below).

Note also that the beta distribution seems to be outside of the scope of the
weighted Poincar\'{e} inequalities described in \cite{MR3519678} since it is
assumed in the latter article that the weight of the considered Poincar\'{e}%
-type inequalities is positive on $\mathbb{R}$, which is not the case for
the beta distribution. Furthermore, \cite{BobkovLedoux:14} also provides
some weighted Cheeger inequality for the Beta distribution, but such a
result seems outside the scope of our approach based on covariance identity (%
\ref{cov_id_kernel}). When considering concentration properties of beta
distributions in Section \ref{section_concentration}\ below, we will however
provide some improvements compared to the results of \cite{BobkovLedoux:14}.

Furthermore, it has also been noticed that the generalized Cauchy
distribution satisfies a weighted Poincar\'{e} distribution, which also
implies in this case a reverse weighted Poincar\'{e} inequality (see 
%TCIMACRO{%
%\TeXButton{\cite{MR2510011, MR3519678}}{\cite{MR2510011, MR3519678}}}%
%BeginExpansion
\cite{MR2510011, MR3519678}%
%EndExpansion
). In fact, \cite{MR2510011} shows that the generalized Cauchy distribution
plays a central role when considering functional inequalities for $\kappa $%
-concave measures, with $\kappa <0$.

The generalized Cauchy distribution $\nu _{\beta }$ of parameter $\beta >1/2$
has density $p_{\beta }\left( x\right) =Z_{\beta }^{-1}\left( 1+x^{2}\right)
^{-\beta }$ for $x\in \mathbb{R}$ and normalizing constant $Z_{\beta }>0$.
Its Stein kernel $\tau _{\beta }$ exists for $\beta >1$ and writes $\tau
_{\beta }\left( x\right) =\left( 1+x^{2}\right) /(2\left( \beta -1\right) )$%
. This allows us to recover in the case where $\beta >3/2$ - that is $\nu
_{\beta }$ has a finite second moment - the optimal weighted Poincar\'{e}
inequality also derived in \cite{MR3519678}, Theorem 3.1. Note that Theorem
3.1 of \cite{MR3519678} also provides the optimal constant in the weighted
Poincar\'{e} inequality with a weight proportional to $1+x^{2}$ in the range 
$\beta \in \left( 1/2,3/2\right] $.

Let us conclude this short list of examples by mentioning Pearson's class of
distributions, for which the density $p$ is solution to the following
differential equation,%
\begin{equation}
\frac{p^{\prime }\left( x\right) }{p\left( x\right) }=\frac{\alpha -x}{\beta
_{2}\left( x-\lambda \right) ^{2}+\beta _{1}\left( x-\lambda \right) +\beta
_{0}}\text{ ,}  \label{def_Pearson}
\end{equation}%
for some constants $\lambda ,\alpha ,\beta _{j},$ $j=0,1,2$. This class of
distributions, that contains for instance Gaussian, Gamma, Beta and Student
distributions, has been well studied in the context of Stein's method, see 
\cite{MR3595350}\ and references therein. In particular, if a density
satisfies (\ref{def_Pearson}) with $\beta _{2}\neq 1/2$, then the
corresponding distribution $\nu $ has a Stein kernel $\tau _{\nu }\left(
x\right) =\left( 1-2\beta _{2}\right) ^{-1}\left( \beta _{0}+\beta
_{1}x+\beta _{2}x^{2}\right) $, for any $x\in I\left( \nu \right) $.
Particularizing to the Student distribution $t_{\alpha }$ with density $%
p_{\alpha }$ proportional to $\left( \alpha +x^{2}\right) ^{-\left( 1+\alpha
\right) /2}$ on $\mathbb{R}$ for $\alpha >1$, we get that for any smooth
function $f\in L_{2}\left( t_{\alpha }\right) $,%
\begin{equation*}
%TCIMACRO{\TeXButton{Var}{\var}}%
%BeginExpansion
\var%
%EndExpansion
_{t_{\alpha }}\left( f\right) \leq \frac{1}{\alpha -1}\int \left(
x^{2}+\alpha \right) f^{\prime 2}\left( x\right) dt_{\alpha }\left( x\right) 
\text{ .}
\end{equation*}

We will investigate concentration inequalities in Section \ref%
{section_concentration}. In fact, existence of a weighted Poincar\'{e}
inequality already implies concentration of measure. The following corollary
is a direct consequence of Theorem 4.1 and Corollary 4.2 in \cite{MR2510011}%
, in light of Theorem \ref{theorem_Poincare_kernel}\ above.

\begin{corollary}
\label{cor_poin_concen}Take a real random variable $X$ of distribution $\nu $
with density $p$ with respect to the Lebesgue measure on $\mathbb{R}$.
Assume that $\mathbb{E}\left[ \left\vert X\right\vert \right] <+\infty $, $p$
has a connected support and denote $\tau _{\nu }$ the Stein kernel of $\nu $%
. Assume that $\sqrt{\tau _{\nu }}$ has a finite $r$th moment, $r\geq 2$.
Then any Lipschitz function $f$ on $\mathbb{R}$ has a finite $r$th moment.
More precisely, if $f$ is $1$-Lipschitz and $\mathbb{E}\left[ f\right] =0$,
then%
\begin{equation}
\left\Vert f\right\Vert _{r}\leq \frac{r}{\sqrt{2}}\left\Vert \sqrt{\tau
_{\nu }}\right\Vert _{r}\text{ .}  \label{moment_bound}
\end{equation}%
Furthermore, by setting $t_{1}=\left\Vert \sqrt{\tau _{\nu }}\right\Vert
_{r}er$, it holds%
\begin{equation*}
\nu \left( \left\vert f\right\vert \geq t\right) \leq \left\{ 
\begin{tabular}{l}
$2e^{-t/(\left\Vert \sqrt{\tau _{\nu }}\right\Vert _{r}e)},$ \\ 
$2\left( \frac{\left\Vert \sqrt{\tau _{\nu }}\right\Vert _{r}r}{t}\right)
^{r},$%
\end{tabular}%
\right. 
\begin{tabular}{l}
if $0\leq t\leq t_{1},$ \\ 
if $t\geq t_{1}.$%
\end{tabular}%
\end{equation*}
\end{corollary}

Inequality (\ref{moment_bound}) of Corollary \ref{cor_poin_concen} may be
compared to Theorem 2.8 in \cite{LedoucNourdinPeccati:15}. Indeed, under the
conditions of Corollary \ref{cor_poin_concen}, Theorem 2.8 in \cite%
{LedoucNourdinPeccati:15} applies and gives the following bound,%
\begin{equation}
\left\Vert f\right\Vert _{r}\leq C\left( S_{r}\left( \nu \left\vert \gamma
\right. \right) +\sqrt{r}\left\Vert \sqrt{\tau _{\nu }}\right\Vert
_{r}\right) \text{ ,}  \label{ineq_LNP}
\end{equation}%
where $S_{r}\left( \nu \left\vert \gamma \right. \right) =\left\Vert \tau
_{\nu }-1\right\Vert _{r}$ is the so-called Stein discrepancy of order $r$
of the measure measure $\nu $ with respect to the normal distribution $%
\gamma $ and $C>0$ is a numerical constant, independent of $r$. Note that
the latter inequality is in fact derived in \cite{LedoucNourdinPeccati:15}
in higher dimension (for a matrix-valued Stein kernel) and is thus from this
point of view much more general than Inequality (\ref{moment_bound}) above.

Considering a measure $\nu $ with a uniformly bounded Stein kernel, as it is
the case for a smooth perturbation of the normal distribution that is
Lipschitz and bounded from zero (see Remark 2.9 in \cite%
{LedoucNourdinPeccati:15}), for a bounded perturbation a measure with
bounded Stein kernel (see Section \ref{section_isoperimetric_constant}\
below for details) or for a strongly log-concave measure as proved in
Section \ref{section_cov_identity}\ above, Inequality (\ref{moment_bound})
only implies a moment growth $\left\Vert f\right\Vert _{r}=O\left( r\right) $
corresponding to exponential concentration of $f$ - which is legitimate
since in this case the weighted Poincar\'{e} inequality (\ref%
{ineq_upper_Laplace_diff})\ is roughly equivalent to a Poincar\'{e}
inequality, the latter classically implying exponential concentration rate -
whereas Inequality (\ref{ineq_LNP}) allows to deduce a sub-Gaussian moment
growth $\left\Vert f\right\Vert _{r}=O\left( \sqrt{r}\right) $. Using a
suitable covariance inequality, we will also recover in Section \ref%
{section_concentration} below a sub-Gaussian concentration rate for a
measure admitting a uniformly bounded Stein kernel.

Notice also that as soon as the moment $\left\Vert \sqrt{\tau _{\nu }}%
\right\Vert _{r}$ grows with $r$, the moment growth of $\left\Vert
f\right\Vert _{r}$ given by (\ref{moment_bound}) is stricly worst than the
rate $O\left( r\right) $ that is achieved by the exponential concentration
rate. However, considering for instance the exponential distribution $d\nu
_{e}=\exp \left( -x\right) \mathbf{1}\left\{ x\in \left[ 0;+\infty \right[
\right\} dx$, its Stein kernel is given by $\tau _{\nu _{e}}\left( x\right)
=x$ on $\mathbb{R}_{+}$ and it is easy to see that Inequality (\ref{ineq_LNP}%
) again allows to recover the right exponential concentration rate, on
contrary to (\ref{moment_bound}) that is sub-optimal in this case. The
concentration of measures having a sub-linear Stein kernel is also
investigated in Section \ref{section_concentration}\ below.

Weakening again the concentration rate, let us consider a generalized Cauchy
measure $\nu _{\beta }$, of density proportional to $\left( 1+x^{2}\right)
^{-\beta }$, $\beta >1$. Its Stein kernel exists and it holds, $\tau _{\nu
_{\beta }}\left( x\right) =\left( 1+x^{2}\right) /(2\left( \beta -1\right) )$%
. In this case, basic computations give that Inequality (\ref{moment_bound})
ensures that for any $r<2\beta -1$, $\left\Vert f\right\Vert _{r}$ is finite
for any centered Lipschitz function $f$, whereas Inequality (\ref{ineq_LNP})
yields a finite estimate only for $r<\beta -1/2$. Consequently, (\ref%
{moment_bound}) allows to consider higher moments than (\ref{ineq_LNP}) for
the generalized Cauchy distributions.

Fathi et al. \cite{courtade2017existence} proved that if a probability
measure on $\mathbb{R}^{d},d\geq 1,$ satisfies a converse weighted Poincar%
\'{e} inequality, then a Stein kernel exists for this measure and the
authors further provide an estimate of the moment of order two of the Stein
kernel under a moment assumption involving the inverse of the weight.

Actually, a slight modification of the arguments provided in \cite%
{courtade2017existence} shows that, reciprocally to Theorem \ref%
{theorem_Poincare_kernel}, if a probability measure satisfies a (direct)
weighted Poincar\'{e} inequality then there exists a natural candidate for
the Stein kernel, in the sense that it satisfies equation (\ref%
{stein_kernel_equation}) on $W_{\nu ,\omega }^{1,2}$, the (weighted) Sobolev
space defined as the closure of all smooth functions in $L_{2}\left( \nu
\right) $ with respect to the Sobolev norm $\int \left( f^{2}+f^{\prime
2}\omega \right) d\nu $. However, to define a Stein kernel, it is usually
required that equation (\ref{stein_kernel_equation}) is satisfied on $W_{\nu
}^{1,2}$, the Sobolev space defined as the closure of all smooth functions
in $L_{2}\left( \nu \right) $ with respect to the Sobolev norm $\int
f^{2}+f^{\prime 2}d\nu $, see for instance \cite{LedoucNourdinPeccati:15}
and \cite{courtade2017existence}. As we can assume without loss of
generality that $\omega \geq 1$, we only have in general $W_{\nu ,\omega
}^{1,2}\subset W_{\nu }^{1,2}$.

\begin{theorem}
\label{theorem_recip_Poin_stein} Assume that a probability measure $\nu $ on 
$\mathbb{R}$ with mean zero and finite second moment satisfies a weighted
Poincar\'{e} inequality (\ref{def_weight_Poin}) with weight $\omega $.
Denote $W_{\nu ,\omega }^{1,2}$ the Sobolev space defined as the closure of
all smooth functions in $L_{2}\left( \nu \right) $ with respect to the
Sobolev norm $\int \left( f^{2}+f^{\prime 2}\omega \right) d\nu $. Then
there exists a function $\tilde{\tau}_{\nu }$, satisfying equation (\ref%
{stein_kernel_equation}) on $W_{\nu ,\omega }^{1,2}$ such that%
\begin{equation}
\int \tilde{\tau}_{\nu }^{2}\omega ^{-1}d\nu \leq \int x^{2}d\nu \text{ .}
\label{ineq_tau_moment_2}
\end{equation}%
If $\tilde{\tau}_{\nu }\in L_{2}\left( \nu \right) $, the weight $\omega $
is locally bounded and the measure has a density $p$ that is also locally
bounded, then $\tilde{\tau}_{\nu }$ is a Stein kernel in the sense that
equation (\ref{stein_kernel_equation}) is satisfied on $W_{\nu }^{1,2}$.
\end{theorem}

One can notice that working with the Stein discrepancy - as for instance in 
\cite{LedoucNourdinPeccati:15} and \cite{courtade2017existence}\ - requires
that the Stein kernel is square integrable. In Theorem \ref%
{theorem_recip_Poin_stein}, we state that if $\tilde{\tau}_{\nu }\in
L_{2}\left( \nu \right) $, the weight $\omega $ is locally bounded and the
measure has a density $p$ that is also locally bounded, then $\tilde{\tau}%
_{\nu }$ is indeed a Stein kernel.

\begin{proof}
The proof is a simple modification of the proof of Theorem 2.4 in \cite%
{courtade2017existence}. We give it for the sake of completeness. Set $%
W_{\nu ,\omega ,0}^{1,2}=W_{\nu ,\omega }^{1,2}\cap \left\{ f\in L_{2}\left(
\nu \right) :\int fd\nu =0\right\} $. Then $\int f^{\prime }g^{\prime
}\omega d\nu $ is a continuous bilinear form on $W_{\nu ,\omega
,0}^{1,2}\times W_{\nu ,\omega ,0}^{1,2}$ and is coercive by the weighted
Poincar\'{e} inequality (\ref{def_weight_Poin}). So by the Lax-Milgram
theorem, there exists a unique $g_{0}\in W_{\nu ,\omega ,0}^{1,2}$ such that%
\begin{equation*}
\int f^{\prime }g_{0}^{\prime }\omega d\nu =\int xfd\nu
\end{equation*}%
for any $f\in W_{\nu ,\omega ,0}^{1,2}$. Hence, $\tau _{\nu }=g_{0}^{\prime
}\omega $ is a Stein kernel for $\nu $. Furthermore, we have%
\begin{eqnarray*}
\int g_{0}^{\prime 2}\omega d\nu &=&\int xg_{0}d\nu \\
&\leq &\sqrt{\int g_{0}^{2}d\nu }\sqrt{\int x^{2}d\nu } \\
&\leq &\sqrt{\int g_{0}^{\prime 2}\omega d\nu }\sqrt{\int x^{2}d\nu }
\end{eqnarray*}%
Noticing that $\omega >0$ $a.s.$ on $I\left( \nu \right) $, we have $\int
g_{0}^{\prime 2}\omega d\nu =\int \tau _{\nu }^{2}\omega ^{-1}d\nu $, which
gives (\ref{ineq_tau_moment_2}). Now, if $\omega $ and $p$ are locally
bounded, then $C_{c}^{\infty }\subset W_{\nu ,\omega }^{1,2}$ where $%
C_{c}^{\infty }$ is the set of infinitely differentiable functions with
compact support. So, by density of $C_{c}^{\infty }$ in $W_{\nu }^{1,2}$
(which since $p$ is locally bounded follows from the same arguments as the
classical density of $C_{c}^{\infty }$ in the usual Sobolev space with
respect to the Lebesgue measure, see for instance Theorem 6.15 in \cite%
{MR1817225}), as $\tilde{\tau}_{\nu }\in L_{2}\left( \nu \right) $ and as $%
\nu $ has a finite second measure, equation (\ref{stein_kernel_equation}) is
valid on $W_{\nu }^{1,2}$.
\end{proof}

Note that for ease of presentation, we stated Theorem \ref%
{theorem_recip_Poin_stein} in dimension one, but it is seen from the proof
above that it directly extends to $\mathbb{R}^{d}$, $d\geq 2$, by
considering the Hilbert-Schmidt scalar product between matrices - since the
Stein kernel is matrix-valued - just as in \cite{courtade2017existence}.

\subsection{Links with Muckenhoupt-type criteria\label{ssection_Muck}}

It is well-known that the Muckenhoupt criterion \cite{MR0311856}, which
provides a necessary and sufficient condition for a (weighted) Hardy
inequalty to hold on the real line, can be used to sharply estimate the best
constant in Poincar\'{e} inequalities (see for instance \cite{MR1845806} and 
\cite{MR2464097} and references therein). The following theorem, providing
sharp estimates, is given in \cite{MR2464097}.

\begin{theorem}[\protect\cite{MR2464097}]
\label{theorem_muck_Poin}Let $\eta $ be a probability measure on $\mathbb{R}$
with median $m$ and let $\chi $ be a measure on $\mathbb{R}$ with
Radon-Nikodym derivative with respect to the Lebesgue measure denoted by $n$%
. The best constant $C_{P}$ such that, for every locally Lipschitz $f:%
\mathbb{R\rightarrow \mathbb{R}}$ it holds%
\begin{equation*}
%TCIMACRO{\TeXButton{Var}{\var}}%
%BeginExpansion
\var%
%EndExpansion
_{\eta }\left( f\right) \leq C_{P}\int \left( f^{\prime }\right) ^{2}d\chi 
\text{ ,}
\end{equation*}%
verifies $\max \left\{ B_{+},B_{-}\right\} \leq C_{P}\leq 4\max \left\{
B_{+},B_{-}\right\} $ where%
\begin{equation*}
B_{+}=\sup_{x>m}\eta \left( \left[ x,+\infty \right) \right) \int_{m}^{x}%
\frac{dt}{n\left( t\right) }\text{ \ \ \ \ and \ \ \ \ }B_{-}=\sup_{x<m}\eta
\left( \left( -\infty ,x\right] \right) \int_{x}^{m}\frac{dt}{n\left(
t\right) }\text{ .}
\end{equation*}
\end{theorem}

Considering Theorem \ref{theorem_Poincare_kernel}, a natural question is:
can we recover (up to a constant) Inequality (\ref{ineq_weighted_Poincare})
from Theorem \ref{theorem_muck_Poin}\ above? The answer is positive. Indeed,
we want to show that $\max \left\{ B_{+},B_{-}\right\} $ is finite. We will
only discuss computations for $B_{+}$ since $B_{-}$ can be treated
symmetrically. With the notations of Theorems \ref{theorem_Poincare_kernel}
and \ref{theorem_muck_Poin}, we take $\eta =\nu $ and $n\left( t\right)
=\tau _{\nu }\left( t\right) p\left( t\right) $. This gives 
\begin{equation*}
B_{+}=\sup_{x>m}\nu \left( \left[ x,+\infty \right) \right) \int_{m}^{x}%
\frac{dt}{\tau _{v}\left( t\right) p\left( t\right) }=\sup_{x>m}\nu \left( %
\left[ x,+\infty \right) \right) \int_{m}^{x}\frac{dt}{\int_{t}^{\infty
}\left( y-\mu \right) p\left( y\right) dy}\text{ .}
\end{equation*}%
Furthermore, there exists $\delta >0$ such that $v\left( \left[
x_{0},+\infty \right) \right) >0$ for $x_{0}=\max \left\{ m,\mu \right\}
+\delta $. Hence, for any $x\geq x_{0},$%
\begin{equation*}
\nu \left( \left[ x,+\infty \right) \right) \int_{m}^{x}\frac{dt}{%
\int_{t}^{\infty }\left( y-\mu \right) p\left( y\right) dy}\leq \frac{\nu
\left( \left[ x,+\infty \right) \right) }{\int_{x}^{\infty }\left( y-\mu
\right) p\left( y\right) dy}\leq \frac{1}{\delta }\text{ .}
\end{equation*}%
As%
\begin{equation*}
\sup_{x\in \left( m,x_{0}\right) }\nu \left( \left[ x,+\infty \right)
\right) \int_{m}^{x}\frac{dt}{\int_{t}^{\infty }\left( y-\mu \right) p\left(
y\right) dy}\leq \frac{x_{0}-m}{\int_{m}^{\infty }\left( y-\mu \right)
p\left( y\right) dy\wedge \int_{x_{0}}^{\infty }\left( y-\mu \right) p\left(
y\right) dy}<+\infty \text{ ,}
\end{equation*}%
we get $B_{+}<+\infty $. Consequently, we quickly recover under the
assumptions of Theorem \ref{theorem_Poincare_kernel} the fact that the
measure $\nu $ satisfies a weighted Poincar\'{e} inequality of the form of (%
\ref{ineq_weighted_Poincare}), although with a multiplicative constant at
the right-hand side that a priori depends on the measure $\nu $. Using (\ref%
{ineq_weighted_Poincare}) together with Theorem \ref{theorem_muck_Poin}, we
have in fact $\max \left\{ B_{+},B_{-}\right\} \leq 1$, but we couldn't
achieve this bound - even up to a numerical constant - by direct
computations. It is maybe worth mentioning that the proof of Theorem \ref%
{theorem_muck_Poin}\ in \cite{MR2464097}\ is technically involved and that
the bound $\max \left\{ B_{+},B_{-}\right\} \leq 1$ in our case might be
rather difficult to establish by direct computations.

Let us turn now to the important, natural question of the existence of
weighted log-Sobolev inequalities under the existence of a Stein kernel. The
proof of the following theorem is based on a Muckenhoupt-type criterion due
to \cite{MR1682772} (see also \cite{MR2052235} for a refinement) giving a
necessary and sufficient condition for existence of (weighted) log-Sobolev
inequalities on $\mathbb{R}$, together with the use of the formulas given in
Propostion \ref{theorem_formulas} above.

\begin{theorem}
\label{log-sob-stein-kernel}Take a real random variable $X$ of distribution $%
\nu $ with density $p$ with respect to the Lebesgue measure on $\mathbb{R}$.
Assume that $\mathbb{E}\left[ \left\vert X\right\vert \right] <+\infty $, $p$
has a connected support $\left[ a,b\right] \subset \mathbb{\bar{R}}$ and
denote $\tau _{\nu }$ the Stein kernel of $\nu $. Take $g$ absolutely
continuous. Then the following inequality holds%
\begin{equation}
%TCIMACRO{\TeXButton{Ent}{\ent}}%
%BeginExpansion
\ent%
%EndExpansion
_{\nu }\left( g^{2}\right) \leq C_{\nu }\int \tau _{\nu }^{2}\left(
g^{\prime }\right) ^{2}d\nu \text{ ,}  \label{log-sob-Stein}
\end{equation}%
for some constant $C_{\nu }>0$ if one of the following aymptotic condition
holds at the supremum of its support, together with one of the symmetric
conditions - that we don't write explicitely since they are obviously
deduced - at the infinimum of its support:

\begin{itemize}
\item $b<+\infty $ and $\tau _{\nu }^{-1}$ is integrable at $b^{-}$ with
respect to the Lebesgue measure.

\item $b=+\infty $ and $0<c_{-}\leq \tau _{\nu }\left( x\right) \leq
c_{+}x^{2}/\log x$ for some constants $c_{-}$and $c_{+}$ and for $x$ large
enough.

\item $b=+\infty $ and $0<c_{-}x\leq \tau _{\nu }\left( x\right) $ for a
constant $c_{-}$and for $x$ large enough.
\end{itemize}

Reciprocally, if $b=+\infty $ and $\tau _{\nu }\rightarrow _{x\rightarrow
\pm \infty }0$ then inequality (\ref{log-sob-Stein}) can not be satisfied
for every smooth function $g$.
\end{theorem}

Theorem \ref{log-sob-stein-kernel} gives a sufficient condition for the
weighted log-Sobolev inequality (\ref{log-sob-Stein}) to hold when the
support is a bounded interval: it suffices that the inverse of the Stein
kernel is integrable at the edges of the support. Furthermore, when an edge
is infinite, if the weighted log-Sobolev inequality (\ref{log-sob-Stein}) is
valid, then the Stein kernel does not tend to zero around this edge.

Consider the Subbotin densities, defined by $p_{\alpha }\left( x\right)
=Z_{\alpha }^{-1}\exp \left( -\left\vert x\right\vert ^{\alpha }/\alpha
\right) $ for $\alpha >0$ and $x\in \mathbb{R}$. There Stein kernels $\tau
_{\alpha }$ satisfy $\tau _{\alpha }\left( x\right) \sim _{x\rightarrow
+\infty }x^{2-\alpha }$ (see the discussion in Section \ref%
{ssection_weighted_Poin} above), so they achieve a weighted log-Sobolev
inequality (\ref{log-sob-Stein}) if and only if $\alpha \in \left( 0,2\right]
$. In the case where $\alpha \in \left[ 2,+\infty \right) $, they actually
achieve a (unweighted) log-Sobolev inequality (\cite{MR1796718}).

It is reasonable to think that inequality (\ref{log-sob-Stein}) should be
true under suitable conditions in higher dimension, with the right-hand side
replaced by $\int \left\vert \tau _{\nu }\nabla f\right\vert ^{2}d\nu $ up
to a constant, where $\tau _{\nu }$ is the Stein kernel constructed by Fathi
in \cite{fathi2018stein} using moment maps. However, some technical details
remain to be solved at this point of our investigations.

Furthermore, notice that Theorem \ref{log-sob-stein-kernel} allows to
recover in dimension one and with a worst constant, a result due to \cite%
{MR2797936} in dimension $d\geq 1$, stating that generalized Cauchy
distributions (see Section \ref{ssection_weighted_Poin}\ above for a
definition) verify the following weighted log-Sobolev inequality,%
\begin{equation}
%TCIMACRO{\TeXButton{Ent}{\ent}}%
%BeginExpansion
\ent%
%EndExpansion
_{\nu _{\beta }}\left( g^{2}\right) \leq \frac{1}{\beta -1}\int \left(
1+x^{2}\right) ^{2}\left( g^{\prime }\left( x\right) \right) ^{2}d\nu
_{\beta }\left( x\right) \text{ , }\beta >1\text{ .}
\label{ineq_log_Sob_Cauchy}
\end{equation}%
Using the Muckenhoupt-type criterion due to \cite{MR1682772}, we can
actually sharpen and extend the previous inequality in the following way.
There exists a constant $C_{\beta }$ such that, for any smooth function $g$,%
\begin{equation}
%TCIMACRO{\TeXButton{Ent}{\ent}}%
%BeginExpansion
\ent%
%EndExpansion
_{\nu _{\beta }}\left( g^{2}\right) \leq C_{\beta }\int \left(
1+x^{2}\right) \log \left( 1+x^{2}\right) \left( g^{\prime }\left( x\right)
\right) ^{2}d\nu _{\beta }\left( x\right) \text{ , }\beta >1/2\text{ .}
\label{ineq_log_Sob_Cauchy_improved}
\end{equation}%
Indeed, using the quantities defined in (\ref{def_Muck_log_Sob}) below, with 
$d\eta =d\nu _{\beta }=Z_{\beta }^{-1}\left( 1+x^{2}\right) ^{-\beta }dx$
and $n\left( t\right) =Z_{\beta }^{-1}\left( 1+x^{2}\right) ^{-\beta +1}\log
\left( 1+x^{2}\right) $, we note that $m=0$, $L_{+}=L_{-}$ and by simple
computations, we have for $\beta >1/2$, 
\begin{equation*}
\int_{m}^{x}\frac{dt}{n\left( t\right) }=O_{x\rightarrow +\infty }\left( 
\frac{x^{2\beta -1}}{\log x}\right) \text{ \ \ \ \ , \ \ \ \ }\Lambda \left(
\nu _{\beta }\left( \left[ x,+\infty \right) \right) \right)
=O_{x\rightarrow +\infty }\left( x^{-2\beta +1}\log x\right) \text{ .}
\end{equation*}%
This means that $L_{+}$ is finite and so, Inequality (\ref%
{ineq_log_Sob_Cauchy_improved}) is valid, which constitutes an improvement
upon inequality (\ref{ineq_log_Sob_Cauchy}). A natural question that remains
open, is wether the weighted log-Sobolev inequality obtained for generalized
Cauchy measures in \cite{MR2797936} can be also improved in dimension $d\geq
2$?

\begin{proof}
Let us denote $\Lambda \left( x\right) =-x\log x$ for $x\geq 0$ - setting by
continuity $\Lambda (0)=0$. By the Muckenhoupt-type criterion derived in 
\cite{MR1682772} (see also Theorem 1 in \cite{MR2052235}), we know that if 
\begin{equation}
L_{+}=\sup_{x>m}\Lambda \left( \eta \left( \left[ x,+\infty \right) \right)
\right) \int_{m}^{x}\frac{dt}{n\left( t\right) }\text{ \ \ \ \ and \ \ \ \ }%
L_{-}=\sup_{x<m}\Lambda \left( \eta \left( \left( -\infty ,x\right] \right)
\right) \int_{x}^{m}\frac{dt}{n\left( t\right) }  \label{def_Muck_log_Sob}
\end{equation}%
are finite, then it holds for any smooth function $g$,%
\begin{equation*}
%TCIMACRO{\TeXButton{Ent}{\ent}}%
%BeginExpansion
\ent%
%EndExpansion
_{\eta }\left( g^{2}\right) \leq C_{LS}\int \left( g^{\prime }\right)
^{2}d\chi \text{ ,}
\end{equation*}%
where $m$ is a median of $\eta $, $n$ is the absolutely continuous component
of $\chi $ and $1/150\leq C_{LS}/\max \left\{ L_{-},L_{+}\right\} \leq 468$.
So by taking $\eta =\nu $ and $n=\tau _{\nu }^{2}p$, we only have to check
that $L_{+}$ and $L_{-}$ are finite. We will only detail computations
related to $L_{+}$ since $L_{-}$ can be handled with symmetric arguments.

Without loss of generality, we may assume that $\nu $ is centered, $\int
xd\nu \left( x\right) =0$. Let us denote $R\left( x\right) =\Lambda \left(
\nu \left( \left[ x,+\infty \right) \right) \right) \int_{m}^{x}\left( \tau
_{\nu }^{2}\left( t\right) p\left( t\right) \right) ^{-1}dt$. First remark
that since $\tau _{\nu }\left( t\right) p\left( t\right) =\int_{t}^{\infty
}\left( y-\mu \right) p\left( y\right) dy$, we have for any $x\in \left(
a,b\right) \dbigcap \left( m,+\infty \right) $,%
\begin{equation*}
\int_{m}^{x}\frac{dt}{n\left( t\right) }=\int_{m}^{x}\frac{p\left( t\right)
dt}{\left( \tau _{\nu }\left( t\right) p\left( t\right) \right) ^{2}}\leq 
\frac{1}{\left( \int_{m}^{\infty }\left( y-\mu \right) p\left( y\right)
dy\right) ^{2}}\vee \frac{1}{\left( \int_{x}^{\infty }\left( y-\mu \right)
p\left( y\right) dy\right) ^{2}}<+\infty \text{ .}
\end{equation*}%
This yieds that for any segment $\left[ c,d\right] \subset \left[ m,b\right) 
$, we have $\sup_{x\in \left[ c,d\right] }R\left( x\right) <+\infty $. We
will thus focus on the behavior of the function $R$ around the supremum $b$
of the support. Recall Formula (\ref{formula_density_2}),%
\begin{equation}
p\left( x\right) \tau _{\nu }\left( x\right) =\mathbb{E}\left[ \left(
X\right) _{+}\right] \exp \left( -\int_{0}^{x}\frac{y}{\mathcal{\tau }_{\nu
}\left( y\right) }dy\right) \text{ .}  \label{formula_exp_density}
\end{equation}%
Moreover, Identity (\ref{formula_tail}) applied to $h=Id$ and $x_{0}=0$,
gives for any $x\in I\left( \nu \right) $, $x>\max \left\{ m,0\right\} $,%
\begin{equation*}
\nu \left( \left[ x,+\infty \right) \right) \leq \mathbb{E}\left[ \left(
X\right) _{+}\right] \frac{1}{x}\exp \left( -\int_{0}^{x}\frac{y}{\mathcal{%
\tau }_{\nu }\left( y\right) }dy\right) \text{ .}
\end{equation*}%
Let $\delta >\max \left\{ m,0\right\} $ be such that $\delta \in I\left( \nu
\right) $ - note that $\delta $ exists since $\nu $ is centered. By using
formula (\ref{formula_exp_density}) and an integration by parts, it holds,
for any $x\in I\left( \nu \right) \cap \left[ \delta ,+\infty \right) $,%
\begin{gather}
\mathbb{E}\left[ \left( X\right) _{+}\right] \int_{m}^{x}\frac{dt}{\tau
_{\nu }^{2}\left( t\right) p\left( t\right) }=\mathbb{E}\left[ \left(
X\right) _{+}\right] \int_{m}^{\delta }\frac{dt}{\tau _{\nu }^{2}\left(
t\right) p\left( t\right) }+\int_{\delta }^{x}\frac{1}{t}\frac{t}{\tau _{\nu
}\left( t\right) }\exp \left( \int_{0}^{t}\frac{y}{\mathcal{\tau }_{\nu
}\left( y\right) }dy\right) dt  \notag \\
\leq \mathbb{E}\left[ \left( X\right) _{+}\right] \int_{m}^{\delta }\frac{dt%
}{\tau _{\nu }^{2}\left( t\right) p\left( t\right) }+\frac{1}{x}\exp \left(
\int_{0}^{x}\frac{y}{\mathcal{\tau }_{\nu }\left( y\right) }dy\right)
+\int_{\delta }^{x}\frac{1}{t^{2}}\exp \left( \int_{0}^{t}\frac{y}{\mathcal{%
\tau }_{\nu }\left( y\right) }dy\right) dt  \label{majo_1}
\end{gather}%
Furthermore, notice that the function $\Lambda $ is increasing - and in
particular bounded - at the neighborhood of $0$ (more precisely on $\left[
0,1/e\right] $). Hence, if $b<+\infty $ and $\tau _{\nu }^{-1}$ is
integrable at $b^{-}$ then it is easily seen from the previous computations
that $\sup_{x>m}\int_{m}^{x}\left( \tau _{\nu }^{2}\left( t\right) p\left(
t\right) \right) ^{-1}dt<+\infty $ and so, $L_{+}$ is finite.

From now on, assume that $b=+\infty $. Hence, inequality (\ref{majo_1}) is
valid with $\delta =1$. If $x$ is large enough, it holds%
\begin{equation*}
\Lambda \left( \nu \left( \left[ x,+\infty \right) \right) \right) \leq
\Lambda \left( \mathbb{E}\left[ \left( X\right) _{+}\right] \frac{\exp
\left( -\int_{0}^{x}\frac{y}{\mathcal{\tau }_{\nu }\left( y\right) }%
dy\right) }{x}\right)
\end{equation*}%
and by using (\ref{majo_1}) we get that there exists a constant $r>0$ such
that, for $x$ large enough,%
\begin{eqnarray}
R\left( x\right) &\leq &\underset{(I)}{r+\underbrace{\frac{1}{x^{2}}\left(
\log \left( \frac{x}{\mathbb{E}\left[ \left( X\right) _{+}\right] }\right)
+\int_{0}^{x}\frac{y}{\mathcal{\tau }_{\nu }\left( y\right) }dy\right) }}
\label{ineq_upper} \\
&&+\underset{(II)}{\underbrace{\left( \log \left( \frac{x}{\mathbb{E}\left[
\left( X\right) _{+}\right] }\right) +\int_{0}^{x}\frac{y}{\mathcal{\tau }%
_{\nu }\left( y\right) }dy\right) \frac{1}{x}\exp \left( -\int_{0}^{x}\frac{y%
}{\mathcal{\tau }_{\nu }\left( y\right) }dy\right) \int_{1}^{x}\frac{1}{t^{2}%
}\exp \left( \int_{0}^{t}\frac{y}{\mathcal{\tau }_{\nu }\left( y\right) }%
dy\right) dt}}\text{ .}  \notag
\end{eqnarray}%
If $\tau _{\nu }\left( x\right) \geq c_{-}>0$ for $x$ large enough, then
quantity $\left( I\right) $ remains bounded around $+\infty $. As for
quantity $(II)$, notice that, for any $a>1$,%
\begin{equation*}
\int_{x/a}^{x}\frac{1}{t^{2}}\exp \left( \int_{0}^{t}\frac{y}{\mathcal{\tau }%
_{\nu }\left( y\right) }dy\right) dt\leq \frac{a-1}{x}\exp \left(
\int_{0}^{x}\frac{y}{\mathcal{\tau }_{\nu }\left( y\right) }dy\right)
\end{equation*}%
and so, for $x$ large enough,%
\begin{equation*}
\int_{1}^{x}\frac{1}{t^{2}}\exp \left( \int_{0}^{t}\frac{y}{\mathcal{\tau }%
_{\nu }\left( y\right) }dy\right) dt\leq \frac{a-1}{x}\exp \left(
\int_{0}^{x}\frac{y}{\mathcal{\tau }_{\nu }\left( y\right) }dy\right)
+\int_{1}^{x/a}\frac{1}{t^{2}}\exp \left( \int_{0}^{t}\frac{y}{\mathcal{\tau 
}_{\nu }\left( y\right) }dy\right) dt\text{ .}
\end{equation*}%
If in addition, $\tau _{\nu }\left( x\right) \leq c_{+}x^{2}/\log x$ for $x$
large enough, then%
\begin{equation*}
\int_{x/a}^{x}\frac{y}{\tau _{\nu }\left( y\right) }dy\geq
c_{+}^{-1}\int_{x/a}^{x}\frac{\log y}{y}dy=c_{+}^{-1}\log a\log
x-c_{+}^{-1}\left( \log a\right) ^{2}\text{ .}
\end{equation*}%
Take $a=\exp \left( c_{+}\right) $, then there exists a constant $%
C_{c_{+}}>0 $ depending on $c_{+}$ such that, for $x$ large enough,%
\begin{equation*}
\int_{1}^{x/a}\frac{1}{t^{2}}\exp \left( \int_{0}^{t}\frac{y}{\mathcal{\tau }%
_{\nu }\left( y\right) }dy\right) dt\leq \exp \left( \int_{0}^{x/a}\frac{y}{%
\mathcal{\tau }_{\nu }\left( y\right) }dy\right) \leq \frac{C_{c_{+}}}{x}%
\exp \left( \int_{0}^{x}\frac{y}{\mathcal{\tau }_{\nu }\left( y\right) }%
dy\right) \text{ .}
\end{equation*}%
This gives that, up to a constant, quantity $\left( II\right) $ is bounded
from above by quantity $\left( I\right) $ and the conclusion follows if $%
c_{-}\leq \tau _{\nu }\left( x\right) \leq c_{+}x^{2}/\log x$ for $x$ large
enough.

We discuss now the case where $c_{-}x\leq \tau _{\nu }\left( x\right) $ for
a constant $c_{-}>0$ and for $x$ large enough. As in particular $\tau _{\nu
}\geq c_{-}>0$ for $x$ large enough, then quantity $(I)$ remains bounded
around $+\infty $. We also have, for $x$ large enough,%
\begin{equation*}
\int_{0}^{x}\frac{y}{\mathcal{\tau }_{\nu }\left( y\right) }dy\leq \frac{2}{%
c_{-}}x
\end{equation*}%
and%
\begin{equation*}
\int_{1}^{x}\frac{1}{t^{2}}\exp \left( \int_{0}^{t}\frac{y}{\mathcal{\tau }%
_{\nu }\left( y\right) }dy\right) dt\leq \exp \left( \int_{0}^{x}\frac{y}{%
\mathcal{\tau }_{\nu }\left( y\right) }dy\right) \text{ ,}
\end{equation*}%
which is sufficient to ensure that $\left( II\right) =O_{x\rightarrow
+\infty }\left( 1\right) $ and $L_{+}<+\infty $.

To conclude the proof, it remains to consider the situation where $b=+\infty 
$ and for any $c>0$ there exists $x_{0}$ such that for any $x\geq x_{0}$, $%
\tau _{\nu }\left( x\right) \leq c$. In this case, we have to show that $%
L_{+}=+\infty $. Since there exists $x_{1}$ such that $\tau _{\nu }\left(
x\right) \leq 1$ for $x_{1}\leq x$, then for $x_{1}\leq x\leq y$,%
\begin{equation*}
\exp \left( -\int_{x}^{y}\frac{t}{\mathcal{\tau }_{\nu }\left( t\right) }%
dy\right) \leq \exp \left( -\frac{y^{2}-x^{2}}{2}\right)
\end{equation*}%
Notice also that, for any $x\geq 2$, we have by integration by parts,%
\begin{equation*}
\int_{x}^{+\infty }\frac{1}{y^{2}}\exp \left( -\frac{y^{2}-x^{2}}{2}\right)
dy\leq \frac{1}{x^{3}}\leq \frac{1}{4x}\text{ .}
\end{equation*}%
Hence, by formula (\ref{formula_tail}) applied to $h=Id$ and $x_{0}=0$, we
have for $x\geq \max \left\{ x_{1},2\right\} $,%
\begin{eqnarray*}
\nu \left( \left[ x,+\infty \right) \right) &=&\mathbb{E}\left[ \left(
X\right) _{+}\right] \exp \left( -\int_{0}^{x}\frac{y}{\mathcal{\tau }_{\nu
}\left( y\right) }dy\right) \left( \frac{1}{x}-\int_{x}^{+\infty }\frac{1}{%
y^{2}}\exp \left( -\int_{x}^{y}\frac{t}{\mathcal{\tau }_{\nu }\left(
t\right) }dt\right) dy\right) \\
&\geq &\mathbb{E}\left[ \left( X\right) _{+}\right] \exp \left( -\int_{0}^{x}%
\frac{y}{\mathcal{\tau }_{\nu }\left( y\right) }dy\right) \left( \frac{1}{x}%
-\int_{x}^{+\infty }\frac{1}{y^{2}}\exp \left( -\frac{y^{2}-x^{2}}{2}\right)
dy\right) \\
&\geq &\frac{3\mathbb{E}\left[ \left( X\right) _{+}\right] }{4x}\exp \left(
-\int_{0}^{x}\frac{y}{\mathcal{\tau }_{\nu }\left( y\right) }dy\right) \text{
.}
\end{eqnarray*}%
It is then easily seen that Inequality (\ref{ineq_upper}) can be reversed,
up to a positive multiplicative constant\ at the right-hand side (together
with changing the value of $r$ that can be negative). Also, if for any $c>0$
there exists $x_{0}$ such that for any $x\geq x_{0}$, $\tau _{\nu }\left(
x\right) \leq c$, then the term $x^{-2}\int_{0}^{x}\mathcal{\tau }_{\nu
}^{-1}\left( y\right) ydy$ goes to $+\infty $ when $x$ tends to $+\infty $.
This gives $L_{+}=+\infty $ and concludes the proof.
\end{proof}

\subsection{Asymmetric Brascamp-Lieb-type inequalities\label%
{ssection_asym_BL}}

The celebrated\ Brascamp-Lieb inequality \cite{MR0450480}\ states that if a
measure $\pi $ on $\mathbb{R}^{d}$, $d\geq 1$, is strictly log-concave -
that is $\varphi =-\ln p$ is convex and $%
%TCIMACRO{\TeXButton{\Hess}{\Hess}}%
%BeginExpansion
\Hess%
%EndExpansion
\left( \varphi \right) \left( x\right) $ is a positive definite symmetric
matrix for any $x\in 
%TCIMACRO{\TeXButton{Supp}{\Supp}}%
%BeginExpansion
\Supp%
%EndExpansion
\left( \pi \right) $ - then for any smooth function $h$,%
\begin{equation*}
%TCIMACRO{\TeXButton{Var}{\var}}%
%BeginExpansion
\var%
%EndExpansion
_{\pi }\left( h\right) \leq \int \nabla h^{T}(%
%TCIMACRO{\TeXButton{\Hess}{\Hess}}%
%BeginExpansion
\Hess%
%EndExpansion
\left( \varphi \right) )^{-1}\nabla hd\pi \text{ .}
\end{equation*}%
Considering the one dimensional case $d=1$, Menz and Otto \cite%
{Menz-Otto:2013} proved that for any smooth $g\in L_{1}\left( \pi \right) $, 
$h\in L_{\infty }\left( \pi \right) $, 
\begin{equation}
\left\vert 
%TCIMACRO{\TeXButton{Cov}{\cov}}%
%BeginExpansion
\cov%
%EndExpansion
_{\pi }\left( g,h\right) \right\vert \leq \left\Vert \frac{h^{\prime }}{%
\varphi ^{\prime \prime }}\right\Vert _{\infty }\left\Vert g^{\prime
}\right\Vert _{1}\text{ .}  \label{asymm_BL}
\end{equation}%
The authors call the later inequality an asymmetric Brascamp-Lieb
inequality. This inequality has been then generalized to higher dimension (%
\cite{CarlenCordero-ErausquinLieb}) and beyond the log-concave case (\cite%
{arnaudon2016intertwinings}).

Considering the covariance of two smooth functions, we will derive in the
next proposition inequalities involving the derivative of these functions
and as well as quantities related to the Stein kernel.

\begin{proposition}
Assume that $\tau _{\nu }>0$ on the support of $\nu $. Then for any $p,q\in %
\left[ 1,+\infty \right] $, $p^{-1}+q^{-1}=1$,%
\begin{equation}
\left\vert 
%TCIMACRO{\TeXButton{Cov}{\cov}}%
%BeginExpansion
\cov%
%EndExpansion
\left( g,h\right) \right\vert \leq \left\Vert g^{\prime }\tau _{\nu
}\right\Vert _{p}\left\Vert \frac{\mathcal{\bar{L}}\left( h\right) }{\tau
_{\nu }}\right\Vert _{q}  \label{ineq_cov_gene_prop}
\end{equation}%
Furthermore, if $p=1$ and $q=+\infty $, we have 
\begin{equation}
\left\vert 
%TCIMACRO{\TeXButton{Cov}{\cov}}%
%BeginExpansion
\cov%
%EndExpansion
\left( g,h\right) \right\vert \leq \left\Vert h^{\prime }\right\Vert
_{\infty }\left\Vert g^{\prime }\tau _{\nu }\right\Vert _{1}\text{ .}
\label{ineq_cov_1_infty_prop}
\end{equation}%
If $p,q\in \left( 1,+\infty \right) $, we also have%
\begin{equation}
\left\vert 
%TCIMACRO{\TeXButton{Cov}{\cov}}%
%BeginExpansion
\cov%
%EndExpansion
\left( g,h\right) \right\vert \leq \left\Vert g^{\prime }\tau _{\nu
}\right\Vert _{p}\left\Vert h^{\prime }m^{1/q}\right\Vert _{q}\text{ ,}
\label{ineq_cov_p_q_prop}
\end{equation}%
where $m\left( x\right) =p\left( x\right) ^{-1}\int k\left( x,y\right) \tau
_{\nu }^{-1}\left( y\right) dy$ on the support of $\nu $. If in addition, $%
\tau _{\nu }\geq \sigma _{\min }^{2}>0$ $a.s.$, then%
\begin{equation}
\left\vert 
%TCIMACRO{\TeXButton{Cov}{\cov}}%
%BeginExpansion
\cov%
%EndExpansion
\left( g,h\right) \right\vert \leq \sigma _{\min }^{-2}\left\Vert g^{\prime
}\tau _{\nu }\right\Vert _{p}\left\Vert h^{\prime }\tau _{\nu
}^{1/q}\right\Vert _{q}\text{ .}  \label{ineq_cov_p_q_2_prop}
\end{equation}
\end{proposition}

Note that if $\nu $ is strongly log-concave, meaning that $\varphi ^{\prime
\prime }\geq c>0$ for $\varphi =-\ln p$, then the asymmetric Brascamp-Lieb
inequality (\ref{asymm_BL}) and covariance identity (\ref%
{ineq_cov_1_infty_prop}) both induce the following inequality,%
\begin{equation*}
\left\vert 
%TCIMACRO{\TeXButton{Cov}{\cov}}%
%BeginExpansion
\cov%
%EndExpansion
\left( g,h\right) \right\vert \leq c^{-1}\left\Vert h^{\prime }\right\Vert
_{\infty }\left\Vert g^{\prime }\right\Vert _{1}\text{ .}
\end{equation*}

\begin{proof}
By (\ref{cov_id_Lbar}) and H\"{o}lder's inequality, we have%
\begin{eqnarray*}
\left\vert 
%TCIMACRO{\TeXButton{Cov}{\cov}}%
%BeginExpansion
\cov%
%EndExpansion
\left( g,h\right) \right\vert &=&\int_{\mathbb{R}}g^{\prime }\tau _{\nu }%
\frac{\mathcal{\bar{L}}h}{\tau _{\nu }}d\nu \\
&\leq &\left\Vert g^{\prime }\tau _{\nu }\right\Vert _{p}\left\Vert \frac{%
\mathcal{\bar{L}}\left( h\right) }{\tau _{\nu }}\right\Vert _{q}\text{ .}
\end{eqnarray*}%
So Inequality (\ref{ineq_cov_gene_prop}) is proved. To prove (\ref%
{ineq_cov_1_infty_prop}) it suffices to apply (\ref{ineq_cov_gene_prop})
with $p=1$ and $q=+\infty $ and remark that%
\begin{equation*}
\left\Vert \frac{\mathcal{\bar{L}}\left( h\right) }{\tau _{\nu }}\right\Vert
_{\infty }=\left\Vert \int h^{\prime }\left( y\right) \frac{k_{\nu }\left(
x,y\right) dy}{\int k_{\nu }\left( x,z\right) dz}\right\Vert _{\infty }\leq
\left\Vert h^{\prime }\right\Vert _{\infty }\text{ .}
\end{equation*}%
Now, to prove (\ref{ineq_cov_p_q_prop}) simply note that%
\begin{equation*}
\left\Vert \frac{\mathcal{\bar{L}}\left( h\right) }{\tau _{\nu }}\right\Vert
_{q}=\left\Vert \int h^{\prime }\left( y\right) \frac{k_{\nu }\left(
x,y\right) dy}{\int k_{\nu }\left( x,z\right) dz}\right\Vert _{q}\leq \int
\int \left( h^{\prime }\left( y\right) \right) ^{q}\frac{k_{\nu }\left(
x,y\right) }{\int k_{\nu }\left( x,z\right) dz}dxdy=\left\Vert h^{\prime
}m^{1/q}\right\Vert _{q}\text{ .}
\end{equation*}%
To deduce (\ref{ineq_cov_p_q_2_prop}) from (\ref{ineq_cov_p_q_prop}), just
remark that%
\begin{equation*}
m\left( x\right) =\frac{1}{p\left( x\right) }\int \frac{k_{\nu }\left(
x,y\right) }{\tau _{\nu }\left( y\right) }dy\leq \frac{1}{\sigma _{\min
}^{2}p\left( x\right) }\int k_{\nu }\left( x,y\right) dy=\frac{\tau _{\nu
}\left( x\right) }{\sigma _{\min }^{2}}\text{ .}
\end{equation*}
\end{proof}

\subsection{Isoperimetric constant\label{section_isoperimetric_constant}}

Let us complete this section about some functional inequalities linked to
the Stein kernel by studying the isoperimetric constant. Recall that for a
measure $\nu $ on $\mathbb{R}^{d}$, an isoperimetric inequality is an
inequality of the form%
\begin{equation}
\nu ^{+}\left( A\right) \geq c\min \left\{ \nu \left( A\right) ,1-\nu \left(
A\right) \right\} \text{ ,}  \label{def_isop}
\end{equation}%
where $c>0$, $A$ is an arbitrary measurable set in $\mathbb{R}^{d}$ and $\nu
^{+}\left( A\right) $ stands for the $\nu $-perimeter of $A$, defined to be%
\begin{equation*}
\nu ^{+}\left( A\right) =\lim \inf_{r\rightarrow 0^{+}}\frac{\nu \left(
A^{r}\right) -\nu \left( A\right) }{r}\text{ ,}
\end{equation*}%
with $A^{r}=\left\{ x\in \mathbb{R}^{d}:\exists a\in A,\text{ }\left\vert
x-a\right\vert <r\right\} $ the $r$-neighborhood of $A$. The optimal value
of $c=Is\left( \nu \right) $ in (\ref{def_isop}) is referred to as the
isoperimetric constant of $\nu $.

The next proposition shows that existence of a uniformly bounded Stein
kernel is essentially sufficient for guaranteeing existence of a positive
isoperimetric constant.

\begin{proposition}
\label{prop_isop}Assume that the probability measure $\nu $ has a connected
support, finite first moment and continuous density $p$ with respect to the
Lebesgue measure. Assume also that its Stein kernel $\tau _{\nu }$ is
uniformly bounded on $I\left( \nu \right) $, $\left\Vert \tau _{\nu
}\right\Vert _{\infty }<+\infty $. Then $\nu $ admits a positive
isoperimetric constant $Is\left( \nu \right) >0$.
\end{proposition}

More precisely, from the proof of Proposition \ref{prop_isop} (see below),
we can extract a quantitative estimate of the isoperimetric constant.
Indeed, under the assumptions of Proposition \ref{prop_isop}, denote by $%
q_{\beta }$ the quantile of order $\beta \in \left( 0,1\right) $ of the
measure $\nu $ and by $\mu $ its mean. Then, for any $\alpha \in \left(
0,1\right) $ such that $q_{\alpha }<\mu <q_{1-\alpha }$, we have%
\begin{equation*}
Is\left( \nu \right) \geq \min \left\{ \alpha ^{-1}\min_{x\in \left[
q_{\alpha },q_{1-\alpha }\right] }\left\{ p\left( x\right) \right\} ,\frac{%
\min \left\{ \mu -q_{\alpha },q_{1-\alpha }-\mu \right\} }{\left\Vert \tau
_{\nu }\right\Vert _{\infty }}\right\} \text{ .}
\end{equation*}

Measures having a uniformly bounded Stein kernel include strongly
log-concave measures - as proved in Section \ref{section_cov_identity}\
above -, but also smooth perturbations of the normal distribution that are
Lipschitz and bounded from below by a positive constant (see Remark 2.9 in 
\cite{LedoucNourdinPeccati:15}). In addition, bounded perturbations of
measures having a bounded Stein kernel given by formula (\ref{kernel_formula}%
) and a density that is bounded away from zero around its mean also have a
bounded Stein kernel. Indeed, take $\varkappa $ a measure having a density $%
p_{\varkappa }$ with connected support $\left[ a,b\right] \subset \mathbb{%
\bar{R}}$, $a<0<b$, mean zero and a finite first moment so that its Stein
kernel $\tau _{\varkappa }$ is unique - up to sets of Lebesgue measure zero
- and given by formula (\ref{kernel_formula}). Assume also that the density $%
p_{\varkappa }$ is bounded away from zero around zero, that is there exists $%
L,\delta >0$ such that $\left[ -\delta ,\delta \right] \subset \left(
a,b\right) $ and $\inf_{x\in \left[ -\delta ,\delta \right] }p_{\varkappa
}\left( x\right) \geq L>0$. Now, consider a function $\rho $ on $\mathbb{R}$
such that $C^{-1}\leq \rho \left( x\right) \leq C$ for any $x\in \left(
a,b\right) $ and for some constant $C>0$, such that $p\left( x\right) =\rho
\left( x\right) p_{\varkappa }\left( x\right) $ is the density of a
probability measure $\nu $ with support $\left[ a,b\right] $. As $\nu $ has
a finite first moment, it also admits a Stein kernel $\tau _{\nu }$, given
by formula (\ref{kernel_formula}). Furthermore, $\tau _{\nu }$ is uniformly
bounded and easy computations using the formula (\ref{kernel_formula}) give%
\begin{equation*}
\left\Vert \tau _{\nu }\right\Vert _{\infty }\leq C^{2}\left( \left\Vert
\tau _{\varkappa }\right\Vert _{\infty }+\left\vert \mu \right\vert \max
\left\{ \frac{1}{L},\frac{\left\Vert \tau _{\varkappa }\right\Vert _{\infty }%
}{\delta }\right\} \right) \text{ ,}
\end{equation*}%
where $\mu $ is the mean of $\nu $.

It would be interesting to know if Proposition \ref{prop_isop} also holds in
higher dimension, but this question remains open. A further natural question
would be: does a measure having a Stein kernel satisfy a weighted
isoperimetric-type inequality, with a weight related to the Stein kernel? So
far, we couldn't give an answer to this question. Note that Bobkov and
Ledoux 
%TCIMACRO{%
%\TeXButton{\cite{MR2797936, MR2510011}}{\cite{MR2797936, MR2510011}} }%
%BeginExpansion
\cite{MR2797936, MR2510011}
%EndExpansion
proved some weighted Cheeger and weighted isoperimetric-type inequalities
for the generalized Cauchy and for $\kappa $-concave distributions.

\begin{proof}
Let $F$ be the cumulative distribution function of $\nu $, $\mu $ be its
mean and let $\varepsilon >0$ be such that $\left[ \mu -\varepsilon ,\mu
+\varepsilon \right] \subset I\left( \nu \right) $. Recall (\cite{BobHoud97}%
, Theorem 1.3) that the isoperimetric constant associated to $\nu $ satisfies%
\begin{equation*}
Is\left( \nu \right) =%
%TCIMACRO{\TeXButton{ess}{\ess}}%
%BeginExpansion
\ess%
%EndExpansion
\inf_{a<x<b}\frac{p\left( x\right) }{\min \left\{ F\left( x\right)
,1-F\left( x\right) \right\} }\text{ ,}
\end{equation*}%
where $a<b$ are the edges of the support of $\nu $. Take $x\in I\left( \nu
\right) $ such that $x-\mu \geq \varepsilon /2$, then%
\begin{eqnarray*}
\tau _{v}\left( x\right) &=&\frac{1}{p\left( x\right) }\int_{x}^{\infty
}\left( y-\mu \right) p\left( y\right) dy \\
&\geq &\varepsilon \frac{1-F\left( x\right) }{2p\left( x\right) }\geq \frac{%
\varepsilon }{2}\frac{\min \left\{ F\left( x\right) ,1-F\left( x\right)
\right\} }{p\left( x\right) }\text{ .}
\end{eqnarray*}%
The same estimate holds for $x\leq \mu -\varepsilon /2$ since $\tau
_{v}\left( x\right) =p\left( x\right) ^{-1}\int_{-\infty }^{x}\left( \mu
-y\right) p\left( y\right) dy$ . Hence,%
\begin{equation}
%TCIMACRO{\TeXButton{ess}{\ess}}%
%BeginExpansion
\ess%
%EndExpansion
\inf_{x\in I\left( \nu \right) ,\text{ }\left\vert x-\mu \right\vert \geq
\varepsilon /2}\frac{p\left( x\right) }{\min \left\{ F\left( x\right)
,1-F\left( x\right) \right\} }\geq \frac{2}{\varepsilon \left\Vert \tau
_{\nu }\right\Vert _{\infty }}>0\text{ .}  \label{ineq_far}
\end{equation}%
Furthermore, we have%
\begin{equation}
\inf_{\left\vert x-\mu \right\vert \leq \varepsilon /2}\frac{p\left(
x\right) }{\min \left\{ F\left( x\right) ,1-F\left( x\right) \right\} }\geq 
\frac{\inf_{\left\vert x-\mu \right\vert \leq \varepsilon /2}p\left(
x\right) }{\min \left\{ F\left( \mu -\varepsilon /2\right) ,1-F\left( \mu
+\varepsilon /2\right) \right\} }>0\text{ .}  \label{ineq_around_mu}
\end{equation}%
The conclusion now follows from combining (\ref{ineq_far}) and (\ref%
{ineq_around_mu}).
\end{proof}

\section{Concentration inequalities\label{section_concentration}}

We state in this section some concentration inequalities related to Stein
kernels in dimension one. Due to the use of covariance identities stated in
Section \ref{section_cov_identity}, the proofs indeed heavily rely on
dimension one. We can however notice that it is known from \cite%
{LedoucNourdinPeccati:15} that a uniformly bounded (multi-dimensional) Stein
kernel ensures a sub-Gaussian concentration rate. We derive sharp
sub-Gaussian inequalities in Theorem \ref{theorem_Mills_type_Gauss}. It is
also reasonable to think that results such as in Theorem \ref%
{theorem_concentration_gene} could be generalized to higher dimension, but
this seems rather nontrivial and is left as an open question.

From Proposition \ref{prop_cov_id}, Section \ref{section_cov_identity}, we
get the following proposition.

\begin{proposition}
\label{theorem_cov_ineq}Assume that $\nu $ has a a finite first moment and a
density $p$ with respect to the Lebesgue measure that has a connected
support. If $g\in L_{\infty }\left( \nu \right) $ and $h\in L_{1}\left( \nu
\right) $ are absolutely continuous and $h$ is $1$-Lipschitz, then%
\begin{equation}
\left\vert 
%TCIMACRO{\TeXButton{Cov}{\cov}}%
%BeginExpansion
\cov%
%EndExpansion
\left( g,h\right) \right\vert \leq \mathbb{E}\left[ \left\vert g^{\prime
}\right\vert \cdot \tau _{v}\right] \text{ ,}  \label{cov_ineq_stein}
\end{equation}%
where $\tau _{v}$ is given in (\ref{kernel_formula}) and is the Stein
kernel. Furthermore, if the Stein kernel is uniformly bounded, that is $\tau
_{v}\in L_{\infty }\left( \nu \right) $, then%
\begin{equation}
\left\vert 
%TCIMACRO{\TeXButton{Cov}{\cov}}%
%BeginExpansion
\cov%
%EndExpansion
\left( g,h\right) \right\vert \leq \left\Vert \tau _{v}\right\Vert _{\infty }%
\mathbb{E}\left[ \left\vert g^{\prime }\right\vert \right] \text{ .}
\label{cov_ineq_bounded_stein}
\end{equation}
\end{proposition}

\begin{proof}
Start from identity (\ref{cov_id_connected}) and simply remark that, for $%
h_{\nu }\left( x\right) =x$,%
\begin{eqnarray*}
\left\vert \mathcal{\mathcal{\bar{L}}}h\left( x\right) \right\vert
&=&\left\vert \frac{1}{p\left( x\right) }\int_{\mathbb{R}}k_{\nu }\left(
x,y\right) h^{\prime }\left( y\right) dy\right\vert \\
&\leq &\frac{\left\Vert h^{\prime }\right\Vert _{\infty }}{p\left( x\right) }%
\int_{\mathbb{R}}k_{\nu }\left( x,y\right) dy \\
&=&\left\Vert h^{\prime }\right\Vert _{\infty }\tau _{v}\left( x\right) 
\text{ .}
\end{eqnarray*}
\end{proof}

Applying techniques similar to those developed in \cite{MR1836739} for
Gaussian vectors (see especially Theorem 2.2), we have the following
Gaussian type concentration inequalities when the Stein kernel is uniformly
bounded.

\begin{theorem}
\label{theorem_Mills_type_Gauss}Assume that $\nu $ has a finite first moment
and a density $p$ with respect to the Lebesgue measure that has a connected
support. Assume also that the Stein kernel $\tau _{v}$ is uniformly bounded, 
$\tau _{v}\in L_{\infty }\left( \nu \right) $, and denote $c=\left\Vert \tau
_{v}\right\Vert _{\infty }^{-1}$. Then the following concentration
inequalities hold. For any $1$-Lipschitz function $g$,%
\begin{equation}
\mathbb{P}\left( g\geq \int gd\nu +r\right) \leq e^{-cr^{2}/2}\text{ .}
\label{Gaussian_type_concen}
\end{equation}%
Furthermore, the function%
\begin{equation*}
T_{g}\left( r\right) =e^{cr^{2}/2}\mathbb{E}\left( g-\mathbb{E}g\right) 
\mathbf{1}_{\left\{ g-\mathbb{E}g\geq r\right\} }
\end{equation*}%
is non-increasing in $r\geq 0$. In particular, for all $r>0$,%
\begin{eqnarray}
\mathbb{P}\left( g-\mathbb{E}g\geq r\right) &\leq &\mathbb{E}\left( g-%
\mathbb{E}g\right) _{+}\frac{e^{-cr^{2}/2}}{r}\text{ ,}
\label{concen_ref_gauss} \\
\mathbb{P}\left( \left\vert g-\mathbb{E}g\right\vert \geq r\right) &\leq &%
\mathbb{E}\left\vert g-\mathbb{E}g\right\vert \frac{e^{-cr^{2}/2}}{r}\text{ .%
}  \label{concen_ref_Gauss_abs}
\end{eqnarray}
\end{theorem}

Inequality (\ref{Gaussian_type_concen}) is closely related to Chatterjee's
Gaussian coupling for random variables with bounded Stein kernel \cite%
{MR2875758}. To our knowledge refined concentration inequalities such as (%
\ref{concen_ref_gauss}) and (\ref{concen_ref_Gauss_abs}) are only available
in the literature for Gaussian random variables or by extension, for
strongly log-concave measures. Indeed, these inequalities can be established
for strongly log-concave measures as an immediate consequence of the
Caffarelli contraction theorem, which states that such measures can be
realized as the pushforward of the Gaussian measure by a Lipschitz function.
We refer to these inequalities as generalized Mills' type inequalities since
taking $g=Id$ in Inequality (\ref{concen_ref_Gauss_abs}) allows to recover
Mills' inequality (see for instance \cite{duembgen2010bounding}): if $Z$ is
the normal distribution, then for any $t>0$, 
\begin{equation*}
\mathbb{P}\left( \left\vert Z\right\vert >t\right) \leq \sqrt{\frac{2}{\pi }}%
\frac{e^{-t^{2}/2}}{t}\text{ .}
\end{equation*}%
Here the setting of a bounded Stein kernel is much larger and include for
instance smooth perturbations of the normal distribution that are Lipschitz
and bounded away from zero (see Remark 2.9 in \cite{LedoucNourdinPeccati:15}%
) or bounded perturbations of measures having a bounded Stein kernel and a
density bounded away from zero around its mean (see Section \ref%
{section_isoperimetric_constant}\ above for more details).

Note that Beta distributions $B_{\alpha ,\beta }$ as defined in (\ref%
{def_beta}) are known to be log-concave of order $\alpha $ whenever $\alpha
\geq 1$ and $\beta \geq 1$, \cite{BobkovLedoux:14}. Using this fact, Bobkov
and Ledoux \cite{BobkovLedoux:14}, Proposition B.10, prove the following
concentration inequality: for $X$ a random variable with distribution $%
B\left( \alpha ,\beta \right) ,$ $\alpha \geq 1,$ $\beta \geq 1$ and any $%
r\geq 0$,%
\begin{equation*}
\mathbb{P}\left( \left\vert X-\mathbb{E}\left( X\right) \right\vert \geq
r\right) \leq 2e^{-\left( \alpha +\beta \right) r^{2}/8}\text{ .}
\end{equation*}

Actually, for any $\alpha ,\beta >0$, the Beta distribution $B\left( \alpha
,\beta \right) $ belongs to the Pearson's class of distributions and its
Stein kernel is given a polynomial of order 2, $\tau _{B\left( \alpha ,\beta
\right) }\left( x\right) =\left( \alpha +\beta \right) ^{-1}x\left(
1-x\right) $ on $\left[ 0,1\right] $ (see \cite{MR3595350}). In particular, $%
\left\Vert \tau _{B\left( \alpha ,\beta \right) }\right\Vert _{\infty
}=4^{-1}\left( \alpha +\beta \right) ^{-1}$ and Theorem \ref%
{theorem_Mills_type_Gauss} applies even in the case where $\alpha ,\beta <1$%
, for which the $B\left( \alpha ,\beta \right) $ distribution is not
log-concave.

\begin{corollary}
\label{corollary_beta}Let $\alpha ,\beta >0$. Take $X$ a random variable
with distribution $B\left( \alpha ,\beta \right) $ and $g$ a $1$-Lipschitz
function on $\left[ 0,1\right] $. Then for all $r>0$,%
\begin{equation*}
\mathbb{P}\left( g\left( X\right) -\mathbb{E}\left[ g\left( X\right) \right]
\geq r\right) \leq \exp \left( -2\left( \alpha +\beta \right) r^{2}\right) 
\text{ .}
\end{equation*}%
Furthermore, for all $r>0$,%
\begin{eqnarray}
\mathbb{P}\left( g\left( X\right) -\mathbb{E}\left[ g\left( X\right) \right]
\geq r\right) &\leq &\mathbb{E}\left( g-\mathbb{E}g\right) _{+}\frac{%
e^{-2\left( \alpha +\beta \right) r^{2}}}{r}\text{ ,}  \notag \\
\mathbb{P}\left( \left\vert g\left( X\right) -\mathbb{E}\left[ g\left(
X\right) \right] \right\vert \geq r\right) &\leq &\mathbb{E}\left\vert g-%
\mathbb{E}g\right\vert \frac{e^{-2\left( \alpha +\beta \right) r^{2}}}{r}
\label{ineq_abs_Mills_beta}
\end{eqnarray}%
and, if $\alpha ,\beta \geq 1$,%
\begin{equation}
\mathbb{P}\left( \left\vert g\left( X\right) -\mathbb{E}\left[ g\left(
X\right) \right] \right\vert \geq r\right) \leq \frac{C}{\sqrt{\alpha +\beta
+1}}\frac{e^{-2\left( \alpha +\beta \right) r^{2}}}{r}\text{ }
\label{ineq_Mills_type_beta}
\end{equation}%
with the value $C=2.5$ - for which we always have $C\left( \alpha +\beta
+1\right) ^{-1/2}<2$ - that holds.
\end{corollary}

\begin{proof}
We only need to prove (\ref{ineq_Mills_type_beta}). Start from (\ref%
{ineq_abs_Mills_beta}). It is sufficient to prove the following inequality,%
\begin{equation*}
\mathbb{E}\left\vert g-\mathbb{E}g\right\vert \leq \frac{2.5}{\sqrt{\alpha
+\beta +1}}\text{ .}
\end{equation*}%
By proposition B.7 of \cite{MR3595350}, by setting $m$ the median of $%
g\left( X\right) $, we have%
\begin{eqnarray*}
\mathbb{E}\left\vert g-m\right\vert &\leq &\frac{2.5}{\sqrt{\alpha +\beta +1}%
}\int_{0}^{1}\sqrt{x\left( 1-x\right) }\left\vert g^{\prime }\left( x\right)
\right\vert dB_{\alpha ,\beta }\left( x\right) \\
&\leq &\frac{2.5}{\sqrt{\alpha +\beta +1}}\int_{0}^{1}\sqrt{x\left(
1-x\right) }dB_{\alpha ,\beta }\left( x\right) \\
&=&\frac{2.5B\left( \alpha +1/2,\beta +1/2\right) }{\sqrt{\alpha +\beta +1}%
B\left( \alpha ,\beta \right) }\text{ .}
\end{eqnarray*}%
Now the conclusion follows from the basic inequalities, $B\left( \alpha
+1/2,\beta +1/2\right) \leq B\left( \alpha ,\beta \right) /2$ and $\mathbb{E}%
\left\vert g-\mathbb{E}g\right\vert \leq 2\mathbb{E}\left\vert
g-m\right\vert $.
\end{proof}

\begin{proof}[Proof of Theorem \protect\ref{theorem_Mills_type_Gauss}]
Take $g$ to be $1$-Lipschitz and mean zero with respect to $\nu $, then for
any $\lambda \geq 0$,%
\begin{eqnarray*}
\mathbb{E}\left[ ge^{\lambda g}\right] &=&%
%TCIMACRO{\TeXButton{Cov}{\cov}}%
%BeginExpansion
\cov%
%EndExpansion
\left( g,e^{\lambda g}\right) \\
&\leq &\left\Vert \tau _{v}\right\Vert _{\infty }\mathbb{E}\left[ \left\vert
\left( e^{\lambda g}\right) ^{\prime }\right\vert \right] \\
&\leq &\frac{\lambda }{c}\mathbb{E}\left[ e^{\lambda g}\right] \text{ .}
\end{eqnarray*}%
Define $J\left( \lambda \right) =\log \mathbb{E}\left[ e^{\lambda g}\right] $%
, $\lambda \geq 0$. We thus have the following differential inequality, $%
J^{\prime }\left( \lambda \right) \leq \lambda /c$. Since $J\left( 0\right)
=0$, this implies that $J\left( \lambda \right) \leq \lambda ^{2}/\left(
2c\right) $. Equivalently, $\mathbb{E}\left[ e^{\lambda g}\right] \leq
e^{\lambda ^{2}/\left( 2c\right) }$, which by the use of Chebyshev's
inequality gives (\ref{Gaussian_type_concen}).

Now, assume that as a random variable $g$ has a continuous positive density $%
p$ on the whole real line. Take $f=U\left( g\right) $ where $U$ is a
non-decreasing (piecewise) differentiable function on $\mathbb{R}$. Applying
(\ref{cov_ineq_bounded_stein}), we get%
\begin{equation}
\mathbb{E}\left[ gU\left( g\right) \right] \leq \mathbb{E}\left[ U^{\prime
}\left( g\right) \right] /c\text{ .}  \label{cov_c_U}
\end{equation}%
Let $G$ be the distribution function of $g$. Given $r>0$ and $\varepsilon >0$%
, applying (\ref{cov_c_U}) to the function $U\left( x\right) =\min \left\{
\left( x-r\right) ^{+},\varepsilon \right\} $ leads to%
\begin{equation*}
\int_{r}^{r+\varepsilon }x\left( x-h\right) dG\left( x\right) +\varepsilon
\int_{r+\varepsilon }^{+\infty }xdG\left( x\right) \leq \frac{G\left(
r+\varepsilon \right) -G\left( r\right) }{c}\text{ .}
\end{equation*}%
Dividing by $\varepsilon $ and letting $\varepsilon $ tend to $0$, we
obtain, for all $h>0$,%
\begin{equation*}
\int_{r}^{\infty }xdG\left( x\right) \leq \frac{p\left( r\right) }{c}\text{ .%
}
\end{equation*}%
Thus, the function $V\left( r\right) =\int_{r}^{+\infty }xdG\left( x\right)
=\int_{r}^{+\infty }xp\left( x\right) dx$ satisfies the differential
inequality $V\left( r\right) \leq -V^{\prime }\left( r\right) /(cr)$, that is%
\begin{equation*}
\left( \log V\left( r\right) \right) ^{\prime }\leq -cr\text{ },
\end{equation*}%
which is equivalent to saying that $\log V\left( r\right) +cr^{2}/2$ is
non-increasing, and therefore the function $T_{g}\left( r\right) $ is
non-increasing.
\end{proof}

We relax now the condition on Stein kernels, by assuming that it is
"sub-linear". This condition is fulfilled by many important distributions,
for instance by the Gaussian, Gamma or Beta distributions. We deduce a
sub-Gamma behavior.

\begin{theorem}
\label{theorem_concentration_gamma}Assume that $\nu $ has a finite first
moment and a density $p$ with respect to the Lebesgue measure that has a
connected support. Assume also that the Stein kernel $\tau _{v}$ is
sub-linear, that is $\tau _{v}(x)\leq a\left\vert x-\mu \right\vert +b$,
where $\mu $ is the mean value of $\nu $. Then for any $1$-Lipschitz
function $g$ and any $r>0$,%
\begin{equation}
\mathbb{P}\left( g\geq \int gd\nu +r\right) \leq e^{-\frac{r^{2}}{2ar+2b}}%
\text{ .}  \label{ineq_sub_Gamma}
\end{equation}
\end{theorem}

When $g=Id$, inequality (\ref{ineq_sub_Gamma}) was proved by Nourdin and
Viens \cite{MR2556018} under the stronger condition that $\tau _{v}(x)\leq
a\left( x-\mu \right) +b$ (which induces that the support of $\nu $ is
bounded from below if $a>0$).

\begin{proof}
Take $g$ to be $1$-Lipschitz and mean zero with respect to $\nu $. Whithout
loss of generality, we may assume that $g$ is bounded (otherwise we
approximate $g$ by thresholding its largest values). Then for any $\lambda
\geq 0$,%
\begin{eqnarray}
\mathbb{E}\left[ ge^{\lambda g}\right] &=&%
%TCIMACRO{\TeXButton{Cov}{\cov}}%
%BeginExpansion
\cov%
%EndExpansion
\left( g,e^{\lambda g}\right)  \notag \\
&\leq &\mathbb{E}\left[ \left\vert \left( e^{\lambda g}\right) ^{\prime
}\right\vert \tau _{v}\right]  \notag \\
&\leq &\lambda \mathbb{E}\left[ e^{\lambda g}\tau _{v}\right] \text{ .}
\label{ineq_1}
\end{eqnarray}%
Furthermore,%
\begin{equation}
\mathbb{E}\left[ e^{\lambda g}\tau _{v}\right] \leq a\mathbb{E}\left[
\left\vert X-\mu \right\vert e^{\lambda g\left( X\right) }\right] +b\mathbb{E%
}\left[ e^{\lambda g}\right]  \label{ineq_sublin_1}
\end{equation}%
and $\mathbb{E}\left[ \left\vert X-\mu \right\vert e^{\lambda g\left(
X\right) }\right] =\mathbb{E}\left[ \left( X-\mu \right) h\left( X\right) %
\right] $ where $h\left( x\right) =\mathrm{\mathrm{sign}}\left( x-\mu
\right) \exp \left( \lambda g\left( x\right) \right) $ and $\mathrm{\mathrm{%
sign}}\left( x-\mu \right) =2\cdot \mathbf{1}\left\{ x\geq \mu \right\} -1$.
As $h^{\prime }\left( x\right) =\mathrm{\mathrm{sign}}\left( x-\mu \right)
\lambda g^{\prime }\left( x\right) \exp \left( \lambda g\left( x\right)
\right) $ $a.s.$, we get%
\begin{equation*}
\mathbb{E}\left[ \left\vert X-\mu \right\vert e^{\lambda g\left( X\right) }%
\right] =\lambda \mathbb{E}\left[ \mathrm{\mathrm{sign}}\left( X-\mu \right)
g^{\prime }\left( X\right) e^{\lambda g\left( X\right) }\tau _{v}\left(
X\right) \right] \leq \lambda \mathbb{E}\left[ e^{\lambda g}\tau _{v}\right] 
\text{ ,}
\end{equation*}%
which gives, by combining with (\ref{ineq_sublin_1}),%
\begin{equation*}
\mathbb{E}\left[ e^{\lambda g}\tau _{v}\right] \leq \lambda a\mathbb{E}\left[
e^{\lambda g}\tau _{v}\right] +b\mathbb{E}\left[ e^{\lambda g}\right] \text{
.}
\end{equation*}%
If $\lambda <1/a$, this gives%
\begin{equation}
\mathbb{E}\left[ e^{\lambda g}\tau _{v}\right] \leq \frac{b}{1-\lambda a}%
\mathbb{E}\left[ e^{\lambda g}\right] \text{ .}  \label{ineq_2}
\end{equation}%
Combining (\ref{ineq_1}) and (\ref{ineq_2}), we obtain, for any $\lambda
<1/a $,%
\begin{equation*}
\mathbb{E}\left[ ge^{\lambda g}\right] \leq \frac{\lambda b}{1-\lambda a}%
\mathbb{E}\left[ e^{\lambda g}\right] \text{ .}
\end{equation*}%
Define $J\left( \lambda \right) =\log \mathbb{E}\left[ e^{\lambda g}\right] $%
, $\lambda \geq 0$. We thus have the following differential inequality, 
\begin{equation*}
J^{\prime }\left( \lambda \right) \leq \frac{\lambda b}{1-\lambda a}.
\end{equation*}%
Since $J\left( 0\right) =0$, this implies that $J\left( \lambda \right) \leq
\lambda ^{2}b/\left( 2(1-\lambda a)\right) $. Equivalently, $\mathbb{E}\left[
e^{\lambda g}\right] \leq e^{\lambda ^{2}b/\left( 2(1-\lambda a)\right) }$,
which by the use of Chebyshev's inequality gives (\ref{Gaussian_type_concen}%
).
\end{proof}

Actually, particularizing to the variable $X$ itself, we have the following
concentration bounds, in the spirit of the generalized Mills' type
inequalities obtained in Theorem \ref{theorem_Mills_type_Gauss}.

\begin{theorem}
\label{theorem_Mills_type_Gamma}Assume that $\nu $ has a finite first moment
and a density $p$ with respect to the Lebesgue measure that has a connected
support. Assume also that the Stein kernel $\tau _{v}$ is "sub-linear", that
is $\tau _{v}(x)\leq a\left\vert x-\mu \right\vert +b$ for some $a>0$ and $%
b\geq 0$, where $\mu $ is the mean value of $\nu $. Then the function%
\begin{equation*}
T\left( r\right) =\left( ar+b\right) ^{-b/a^{2}}e^{r/a}\mathbb{E}\left[
\left( X-\mu \right) \mathbf{1}_{\left\{ X-\mu \geq r\right\} }\right]
\end{equation*}%
is non-increasing in $r\geq 0$. In particular, for all $r>0$,%
\begin{eqnarray}
\mathbb{P}\left( X\geq \mu +r\right) &\leq &\mathbb{E}\left( X-\mu \right)
_{+}\frac{\left( ar+b\right) ^{b/a^{2}}e^{-r/a}}{r}\text{ ,}
\label{ineq_Mills_right_gamma} \\
\mathbb{P}\left( \left\vert X-\mu \right\vert \geq r\right) &\leq &\mathbb{E}%
\left\vert X-\mu \right\vert \frac{\left( ar+b\right) ^{b/a^{2}}e^{-r/a}}{r}%
\text{ .}  \label{ineq_Mills_left_gamma}
\end{eqnarray}
\end{theorem}

The concentration bounds (\ref{ineq_Mills_right_gamma})\ and (\ref%
{ineq_Mills_left_gamma}), that seem to be new, have an interest for large
values of $r$, where they improve upon Theorem \ref%
{theorem_concentration_gamma} if $a>0$, due to the factor $2$ in front of
the constant $a$ in the right-hand side of (\ref{ineq_sub_Gamma}).

\begin{proof}
The proof is essentially an adaptation of the proof of Theorem \ref%
{theorem_Mills_type_Gauss}. Without loss of generality, assume that the
support of $p$ is $\mathbb{R}$ and denote $g\left( X\right) =X-\mu $ where $%
X\sim \nu $. Take $f=U\left( g\right) $ where $U$ is a non-decreasing
piecewise differentiable function on $\mathbb{R}$. Applying (\ref%
{cov_ineq_stein}), we get%
\begin{equation}
\mathbb{E}\left[ gU\left( g\right) \right] \leq \mathbb{E}\left[ \tau _{\nu
}U^{\prime }\left( g\right) \right] \leq a\mathbb{E}\left[ \left\vert
g\right\vert U^{\prime }\left( g\right) \right] +b\mathbb{E}\left[ U^{\prime
}\left( g\right) \right] \text{ .}  \label{ineq_cov_U}
\end{equation}%
Let $G$ be the distribution function of $g\left( X\right) $. Given $r>0$ and 
$\varepsilon >0$, applying (\ref{ineq_cov_U}) to the function $U\left(
x\right) =\min \left\{ \left( x-r\right) _{+},\varepsilon \right\} $ leads to%
\begin{equation*}
\int_{r}^{r+\varepsilon }x\left( x-r\right) dG\left( x\right) +\varepsilon
\int_{r+\varepsilon }^{+\infty }xdG\left( x\right) \leq
\int_{r}^{r+\varepsilon }(ax+b)dG\left( x\right) \text{ .}
\end{equation*}%
Dividing by $\varepsilon $ and letting $\varepsilon $ tend to $0$, we
obtain, for all $r>0$,%
\begin{equation*}
\int_{r}^{\infty }xdG\left( x\right) \leq (ar+b)p\left( r+\mu \right) \text{
.}
\end{equation*}%
Thus, the function $V\left( r\right) =\int_{r}^{+\infty }xdG\left( x\right)
=\int_{r}^{+\infty }xp\left( x+\mu \right) dx$ satisfies the differential
inequality $V\left( r\right) \leq -(a+\frac{b}{r})V^{\prime }\left( r\right) 
$, that is%
\begin{equation*}
\left( \log V\left( r\right) \right) ^{\prime }\leq -\frac{r}{ar+b}=\frac{b}{%
a^{2}}\frac{a}{ar+b}-\frac{1}{a}\text{ },
\end{equation*}%
which is equivalent to saying that $\log V\left( r\right) +r/a-ba^{-2}\ln
\left( ar+b\right) $ is non-increasing, and therefore the function $T\left(
r\right) $ is non-increasing on $\mathbb{R}_{+}$. This directly gives (\ref%
{ineq_Mills_right_gamma}). By applying (\ref{ineq_Mills_right_gamma}) to the
deviations at the right of the mean for $-X$, we also get%
\begin{equation*}
\mathbb{P}\left( X\leq \mu -r\right) \leq \mathbb{E}\left( X-\mu \right) _{-}%
\frac{\left( ar+b\right) ^{b/a^{2}}e^{-r/a}}{r}\text{ .}
\end{equation*}%
Combining the last inequality with (\ref{ineq_Mills_right_gamma}) yields (%
\ref{ineq_Mills_left_gamma}).
\end{proof}

Let us now state a more general theorem.

\begin{theorem}
\label{theorem_concentration_gene}Assume that $\nu $ has a finite first
moment, a density $p$ with respect to the Lebesgue measure that has a
connected support and denote $\tau _{\nu }$ its Stein kernel. Set $X$ a
random variable of distribution $\nu $. Take $f$ a $1$-Lipschitz function
with mean zero with respect to $\nu $ and assume that $f$ has an exponential
moment with respect to $\nu $, that is there exists $a>0$ such that $\mathbb{%
E}\left[ e^{af\left( X\right) }\right] <+\infty $. Then for any $\lambda \in
\left( 0,a\right) $,%
\begin{equation}
\mathbb{E}\left[ e^{\lambda f\left( X\right) }\right] \leq \mathbb{E}\left[
e^{\lambda ^{2}\tau _{\nu }\left( X\right) }\right] \text{ .}
\label{upper_Laplace}
\end{equation}%
Consequently, if we denote $\psi _{\tau }\left( \lambda \right) =\ln \mathbb{%
E}\left[ e^{\lambda ^{2}\tau _{\nu }\left( X\right) }\right] \in \left[
0,+\infty \right] $ and $\psi _{\tau }^{\ast }\left( t\right) =\sup_{\lambda
\in \left( 0,a\right) }\left\{ t\lambda -\psi _{\tau }\left( \lambda \right)
\right\} $ the Fenchel-Legendre dual function of $\psi _{\lambda }$, then
for any $r>0$,%
\begin{equation}
\mathbb{P}\left( f\left( X\right) >r\right) \vee \mathbb{P}\left( f\left(
X\right) <-r\right) \leq \exp \left( -\psi _{\tau }^{\ast }\left( r\right)
\right) \text{ .}  \label{ineq_concen_gene}
\end{equation}
\end{theorem}

Theorem \ref{theorem_concentration_gene} states that the concentration of
Lipschitz functions taken on a real random variable with existing Stein
kernel is controlled by the behavior of the exponential moments of the Stein
kernel itself - if it indeed admits finite exponential moments.

Let us now briefly detail how to recover from Theorem \ref%
{theorem_concentration_gene} some results of Theorems \ref%
{theorem_Mills_type_Gauss} and \ref{theorem_concentration_gamma}, although
with less accurate constants. If $\left\Vert \tau _{\nu }\right\Vert
_{\infty }<+\infty $, then inequality (\ref{upper_Laplace}) directly implies%
\begin{equation*}
\mathbb{E}\left[ e^{\lambda f\left( X\right) }\right] \leq e^{\lambda
^{2}\left\Vert \tau _{\nu }\right\Vert _{\infty }}\text{ ,}
\end{equation*}%
which gives%
\begin{equation*}
\mathbb{P}\left( f\left( X\right) >r\right) \vee \mathbb{P}\left( f\left(
X\right) <-r\right) \leq \exp \left( -\frac{r^{2}}{4\left\Vert \tau _{\nu
}\right\Vert _{\infty }}\right) \text{ .}
\end{equation*}%
The latter inequality takes the form of Inequality (\ref%
{Gaussian_type_concen}) of Theorem \ref{theorem_Mills_type_Gauss}, although
with a factor $1/2$ in the argument of the exponential in the right-hand
side of the inequality.

Assume now, as in Theorem \ref{theorem_concentration_gamma}, that the Stein
kernel $\tau _{\nu }$ is sub-linear, that is there exist $a,b\in \mathbb{R}%
_{+}$ such that $\tau _{\nu }\left( x\right) \leq a\left( x-\mu \right) +b$,
where $\mu $ is the mean value of $\nu $. Inequality (\ref{upper_Laplace})
implies in this case,%
\begin{equation}
\mathbb{E}\left[ e^{\lambda f\left( X\right) }\right] \leq \mathbb{E}\left[
e^{a\lambda ^{2}\left( X-\mu \right) }\right] e^{b\lambda ^{2}}\text{ .}
\label{upper_Laplace_Gamma}
\end{equation}%
The latter inequality being valid for any $f$ being $1$-Lipschitz and
centered with respect to $\nu $, we can apply it for $f\left( X\right)
=X-\mu $. This gives%
\begin{equation*}
\mathbb{E}\left[ e^{\lambda \left( X-\mu \right) }\right] \leq \mathbb{E}%
\left[ e^{a\lambda ^{2}\left( X-\mu \right) }\right] e^{b\lambda ^{2}}\text{
.}
\end{equation*}%
Now, considering $\lambda <a^{-1}$, we have by H\"{o}lder's inequality, $%
\mathbb{E}\left[ e^{a\lambda ^{2}\left( X-\mu \right) }\right] \leq \mathbb{E%
}\left[ e^{\lambda \left( X-\mu \right) }\right] ^{a\lambda }$. Plugging
this estimate in the last inequality and rearranging the terms of the
inequality gives%
\begin{equation*}
\mathbb{E}\left[ e^{\lambda \left( X-\mu \right) }\right] \leq e^{\frac{%
b\lambda ^{2}}{1-\lambda a}}\text{ .}
\end{equation*}%
Going back to inequality (\ref{upper_Laplace_Gamma}), we obtain, for any $%
\lambda \in \left( 0,a^{-1}\right) $,%
\begin{equation*}
\mathbb{E}\left[ e^{\lambda f\left( X\right) }\right] \leq \mathbb{E}\left[
e^{\lambda \left( X-\mu \right) }\right] ^{a\lambda }e^{b\lambda ^{2}}\leq
e^{b\lambda ^{2}\left( \frac{\lambda a}{1-\lambda a}+1\right) }=e^{\frac{%
b\lambda ^{2}}{1-\lambda a}}\text{ .}
\end{equation*}%
By the use of Cram\`{e}r-Chernoff method, this gives the result of Theorem %
\ref{theorem_concentration_gamma}, although with a constant $1/2$ in the
argument of the exponential term controlling the deviations.

\begin{proof}
First note that Inequality (\ref{ineq_concen_gene}) is direct consequence of
Inequality (\ref{upper_Laplace}) via the use of the Cram\`{e}r-Chernoff
method (see for instance Section 2.2 of \cite{MR3185193}). To prove
Inequality (\ref{upper_Laplace}), also note that by Lemma \ref%
{Lemma_upper_Laplace} below, it suffices to prove that for any $\lambda \in
\left( 0,a\right) $,%
\begin{equation}
\mathbb{E}\left[ \lambda f\left( X\right) e^{\lambda f\left( X\right) }%
\right] \leq \mathbb{E}\left[ \lambda ^{2}\tau _{\nu }\left( X\right)
e^{\lambda f\left( X\right) }\right] \text{ }.
\label{ineq_upper_Laplace_diff}
\end{equation}%
Take $\lambda \in \left( 0,a\right) $, it holds by identity (\ref%
{cov_id_connected}), 
\begin{eqnarray*}
\mathbb{E}\left[ f\left( X\right) e^{\lambda f\left( X\right) }\right] &=&%
%TCIMACRO{\TeXButton{Cov}{\cov}}%
%BeginExpansion
\cov%
%EndExpansion
\left( f\left( X\right) ,e^{\lambda f\left( X\right) }\right) \\
&=&\mathbb{E}\left[ \lambda f^{\prime }\left( X\right) \mathcal{\bar{L}}%
\left( \lambda f\right) \left( X\right) e^{\lambda f\left( X\right) }\right] 
\text{ .}
\end{eqnarray*}%
Hence, we obtain%
\begin{equation*}
\mathbb{E}\left[ f\left( X\right) e^{\lambda f\left( X\right) }\right] \leq
\lambda ^{2}\mathbb{E}\left[ \left\vert f^{\prime }\left( X\right)
\right\vert \tau _{\nu }\left( X\right) e^{\lambda f\left( X\right) }\right]
\leq \mathbb{E}\left[ \lambda ^{2}\tau _{\nu }\left( X\right) e^{\lambda
f\left( X\right) }\right] \text{ .}
\end{equation*}%
Inequality (\ref{ineq_upper_Laplace_diff}) is thus proved, which completes
the proof.
\end{proof}

\begin{lemma}
\label{Lemma_upper_Laplace} Take $X$ a random variable on a measurable space 
$\left( \mathcal{X},\mathcal{T}\right) $. Take $g$ and $h$ two measurable
functions from $\mathcal{X}$ to $\mathbb{R}$ such that%
\begin{equation}
\mathbb{E}\left[ g\left( X\right) e^{g\left( X\right) }\right] \leq \mathbb{E%
}\left[ h\left( X\right) e^{g\left( X\right) }\right] <+\infty \text{ .}
\label{Assump_Lem}
\end{equation}%
Then it holds,%
\begin{equation}
\mathbb{E}\left[ e^{g\left( X\right) }\right] \leq \mathbb{E}\left[
e^{h\left( X\right) }\right] \text{ .}  \label{ineq_Laplace_Lem}
\end{equation}
\end{lemma}

Lemma \ref{Lemma_upper_Laplace} summarizes the essence of the argument used
in the proof of Theorem 2.3 of \cite{MR1836739}. We could not find a
reference in the literature for Lemma \ref{Lemma_upper_Laplace}. We point
out that Lemma \ref{Lemma_upper_Laplace} may have an interest by itself as
it should be very handy when dealing with concentration inequalities using
the Cram\`{e}r-Chernoff method. Its scope may thus go beyond our framework
related to the behavior of the Stein kernel.

\begin{proof}
Note that if $\mathbb{E}\left[ e^{h\left( X\right) }\right] =+\infty $ then
Inequality (\ref{ineq_Laplace_Lem}) is satisfied. We assume now that $%
\mathbb{E}\left[ e^{h\left( X\right) }\right] <+\infty $ and $\beta =\ln
\left( \mathbb{E}\left[ e^{h\left( X\right) }\right] \right) $. By setting $%
U=h\left( X\right) -\beta $, we get $\mathbb{E}\left[ e^{U}\right] =1$ and
so, by the duality formula for the entropy (see for instance Theorem 4.13 in 
\cite{MR3185193}), we have%
\begin{equation*}
\mathbb{E}\left[ Ue^{g\left( X\right) }\right] \leq 
%TCIMACRO{\TeXButton{Ent}{\ent}}%
%BeginExpansion
\ent%
%EndExpansion
\left( e^{g\left( X\right) }\right) =\mathbb{E}\left[ g\left( X\right)
e^{g\left( X\right) }\right] -\mathbb{E}\left[ e^{g\left( X\right) }\right]
\ln \left( \mathbb{E}\left[ e^{g\left( X\right) }\right] \right) \text{ .}
\end{equation*}%
Furthermore,%
\begin{equation*}
\mathbb{E}\left[ g\left( X\right) e^{g\left( X\right) }\right] -\beta 
\mathbb{E}\left[ e^{g\left( X\right) }\right] \leq \mathbb{E}\left[
Ue^{g\left( X\right) }\right] \text{ .}
\end{equation*}%
Putting the above inequalities together, we obtain $\beta \geq \ln \left( 
\mathbb{E}\left[ e^{g\left( X\right) }\right] \right) $, which is equivalent
to (\ref{ineq_Laplace_Lem}).
\end{proof}

\section{Tail bounds\label{section_density_formula_tail_bounds}}

In the following theorem, we establish lower tail bounds when the Stein
kernel is uniformly bounded away from zero. In particular, the support of
the measure is $\mathbb{R}$ in this case, as can be seen from the explicit
formula (\ref{kernel_formula}).

\begin{theorem}
\label{prop_kernel_lower_bound}Take a real random variable $X$ of
distribution $\nu $ with density $p$ with respect to the Lebesgue measure on 
$\mathbb{R}$. Assume that $\mathbb{E}\left[ X\right] =0$, $p$ has a
connected support and denote $\tau _{\nu }$ the Stein kernel of $\nu $. If $%
\tau _{\nu }\geq \sigma _{\min }^{2}>0$ $\nu $-almost surely, then the
density $p$ of $\nu $ is positive on $\mathbb{R}$ and the function%
\begin{equation*}
R\left( x\right) =e^{x^{2}/2\sigma _{\min }^{2}}\int_{x}^{+\infty }yp\left(
y\right) dy
\end{equation*}%
is nondecreasing on $\mathbb{R}_{+}$. In particular, for any $x\geq 0$,%
\begin{equation}
\int_{x}^{+\infty }yp\left( y\right) dy\geq \mathbb{E}\left[ (X)_{+}\right]
e^{-x^{2}/2\sigma _{\min }^{2}}\text{ .}  \label{mino_tail_unif}
\end{equation}%
By symmetry, for any $x\leq 0$,%
\begin{equation}
_{-}\int_{-\infty }^{x}yp\left( y\right) dy\geq \mathbb{E}\left[ (X)_{-}%
\right] e^{-x^{2}/2\sigma _{\min }^{2}}\text{ .}  \label{mino_tail_unif_left}
\end{equation}%
Assume in addition that the function $L\left( x\right) =x^{1+\beta }p\left(
x\right) $ is nonincreasing on $\left[ s,+\infty \right) ,$ $s>0$. Then for
all $x\geq s$, it holds%
\begin{equation}
\mathbb{P}\left( X\geq x\right) \geq \left( 1-\frac{1}{\beta }\right) \frac{%
\mathbb{E}\left[ \left( X\right) _{+}\right] }{x}\exp \left( -\frac{x^{2}}{%
2\sigma _{\min }^{2}}\right) \text{ .}  \label{mino_tail_poly}
\end{equation}%
Alternatively, assume that there exists $\alpha \in \left( 0,2\right) $ such
that $\lim \sup_{x\rightarrow +\infty }x^{-\alpha }\log \tau _{\nu }\left(
x\right) <+\infty $. Then for any $\delta \in \left( 0,2\right) $, there
exist $L,x_{0}>0$ such that, for all $x>x_{0}$,%
\begin{equation}
\mathbb{P}\left( X\geq x\right) \geq \frac{L}{x}\exp \left( -\frac{x^{2}}{%
\left( 2-\delta \right) \sigma _{\min }^{2}}\right) \text{ .}
\label{mino_tail_limsup}
\end{equation}
\end{theorem}

The results presented in Theorem \ref{prop_kernel_lower_bound} can be found
in \cite{MR2556018} under the additional assumption, related to the use of
Malliavin calculus, that the random variable $X\in \mathbf{D}^{1,2}$.

\begin{proof}
For any smooth function $\varphi $ nondecreasing,%
\begin{equation*}
\mathbb{E}\left[ X\varphi \left( X\right) \right] =\mathbb{E}\left[ \tau
_{\nu }\left( X\right) \varphi ^{\prime }\left( X\right) \right] \geq \sigma
_{\min }^{2}\mathbb{E}\left[ \varphi ^{\prime }\left( X\right) \right] \text{
.}
\end{equation*}%
Take $\varphi \left( x\right) =\min \left\{ \left( x-c\right)
_{+},\varepsilon \right\} $, for some $c\in \mathbb{R}$ and $\varepsilon >0$%
. Then%
\begin{equation*}
\mathbb{E}\left[ X\varphi \left( X\right) \right] =\int_{c}^{c+\varepsilon
}x\left( x-c\right) p\left( x\right) dx+\varepsilon \int_{c+\varepsilon
}^{+\infty }xp\left( x\right) dx
\end{equation*}%
and $\mathbb{E}\left[ \varphi ^{\prime }\left( X\right) \right] =\mathbb{P}%
\left( X\in \left( c,c+\varepsilon \right] \right) $. Dividing the latter
two terms by $\varepsilon $ and letting $\varepsilon $ tend to zero gives,%
\begin{equation}
\int_{c}^{+\infty }xp\left( x\right) dx\geq \sigma _{\min }^{2}p\left(
c\right) \text{ .}  \label{ineq_diff}
\end{equation}%
Now set $V\left( c\right) =\int_{c}^{+\infty }xp\left( x\right) dx$.
Inequality (\ref{ineq_diff}) writes, for any $c\geq 0$,%
\begin{equation*}
\frac{c}{\sigma _{\min }^{2}}V\left( c\right) \geq -V^{\prime }\left(
c\right) \text{ .}
\end{equation*}%
Then define, for any $c\geq 0$,%
\begin{equation*}
R\left( c\right) =V\left( c\right) \exp \left( \frac{c^{2}}{2\sigma _{\min
}^{2}}\right) \text{ .}
\end{equation*}%
We can differentiate $R$ and we have%
\begin{equation*}
R^{\prime }\left( c\right) =\left( V^{\prime }\left( c\right) +\frac{c}{%
\sigma _{\min }^{2}}V\left( c\right) \right) \exp \left( \frac{c^{2}}{%
2\sigma _{\min }^{2}}\right) \geq 0\text{ .}
\end{equation*}%
In particular $R\left( c\right) \geq R\left( 0\right) $, which gives (\ref%
{mino_tail_unif}). As $\tau _{-X}\left( x\right) =\tau _{X}\left( -x\right) $%
, we deduce by symmetry that (\ref{mino_tail_unif_left}) also holds. The
proof of inequalities (\ref{mino_tail_poly}) and (\ref{mino_tail_limsup})
follows from the same arguments as in the proof of points (ii) and (ii)',
Theorem 4.3, \cite{MR2556018}. We give them, with slight modifications, for
the sake of completeness. By integration by parts, we have%
\begin{equation*}
V\left( c\right) =c\mathbb{P}\left( X\geq c\right) +\int_{c}^{+\infty }%
\mathbb{P}\left( X\geq x\right) dx\text{ .}
\end{equation*}%
We also have, for $x>0$,%
\begin{equation*}
\mathbb{P}\left( X\geq x\right) =\int_{x}^{+\infty }\frac{y^{1+\beta
}p\left( y\right) }{y^{1+\beta }}dy\leq x^{1+\beta }p\left( x\right)
\int_{x}^{+\infty }\frac{dy}{y^{1+\beta }}=\frac{xp\left( x\right) }{\beta }%
\text{ .}
\end{equation*}%
Hence,%
\begin{equation*}
V\left( c\right) \leq c\mathbb{P}\left( X\geq c\right) +\beta
^{-1}\int_{c}^{+\infty }xp\left( x\right) dx=c\mathbb{P}\left( X\geq
c\right) +\frac{V\left( c\right) }{\beta }
\end{equation*}%
or equivalently,%
\begin{equation*}
\mathbb{P}\left( X\geq c\right) \geq \left( 1-\frac{1}{\beta }\right) \frac{%
V\left( c\right) }{c}\text{ .}
\end{equation*}%
The conclusion follows by combining the latter inequality with inequality (%
\ref{mino_tail_unif}). It remains to prove (\ref{mino_tail_limsup}). By
formula (\ref{formula_tail}) applied with $h\left( y\right) \equiv y$ - note
that this is possible since by assumption $\tau _{\nu }>0$ on $\mathbb{R}$
-, it holds%
\begin{eqnarray*}
p\left( x\right) &=&\frac{\mathbb{E}\left[ \left\vert X\right\vert \right] }{%
2\tau _{\nu }\left( x\right) }\exp \left( -\int_{0}^{x}\frac{y}{\tau _{\nu
}\left( y\right) }dy\right) \\
&\geq &\frac{\mathbb{E}\left[ \left\vert X\right\vert \right] }{2\tau _{\nu
}\left( x\right) }\exp \left( -\frac{x^{2}}{2\sigma _{\min }^{2}}\right) 
\text{ .}
\end{eqnarray*}%
Let us fix $\varepsilon >0$. By assumption on $\tau _{\nu }$, we get that
there exists a positive constant $C$ such that, for $x$ large enough, 
\begin{equation*}
p\left( x\right) \geq C\exp \left( -\frac{x^{2}}{2\sigma _{\min }^{2}}%
-x^{\alpha }\right) \geq C\exp \left( -\frac{x^{2}}{(2-\varepsilon )\sigma
_{\min }^{2}}\right) \text{ .}
\end{equation*}%
Hence, for $x$ large enough,%
\begin{equation*}
\mathbb{P}\left( X\geq x\right) \geq C\int_{x}^{+\infty }\exp \left( -\frac{%
y^{2}}{(2-\varepsilon )\sigma _{\min }^{2}}\right) dy\text{ .}
\end{equation*}%
The conclusion now easily follows from the following classical inequality: $%
\int_{x}^{+\infty }e^{-y^{2}/2}dy\geq (x/(1+x^{2}))\exp \left(
-x^{2}/2\right) $.
\end{proof}

In the following proposition, we give some further tail bounds under a
variety of assumptions on the Stein kernel. We omit the proof as it follows
directly from the same arguments as in the proof of Corollary 4.5 in \cite%
{MR2568291}, where they are derived under the assumption that $X\in \mathbf{D%
}^{1,2}$.

\begin{proposition}
Take a real random variable $X$ of distribution $\nu $ with density $p$ with
respect to the Lebesgue measure on $\mathbb{R}$. Assume that $\mathbb{E}%
\left[ \left\vert X\right\vert \right] <+\infty $, $p$ has a connected
support and denote $\tau _{\nu }$ the Stein kernel of $\nu $. If there exist 
$c\in \left( 0,1\right) $ and $x_{0}>1$ such that for all $x>x_{0}$, $\tau
_{\nu }\left( x\right) \leq cx^{2}$, then there exists a positive constant $%
L $ such that, for all $x>x_{0}$%
\begin{equation*}
\mathbb{P}\left( X\geq x\right) \geq L\frac{\exp \left( -\int_{0}^{x}\frac{y%
}{\tau _{\nu }\left( y\right) }dy\right) }{x}\text{ .}
\end{equation*}%
If in addition $\tau _{\nu }\geq \sigma _{\min }^{2}>0$ $\nu $-almost
surely, then%
\begin{equation*}
\mathbb{P}\left( X\geq x\right) \geq \frac{L}{x}\exp \left( -\frac{x^{2}}{%
2\sigma _{\min }^{2}}\right) \text{ .}
\end{equation*}%
If there exists rather a positive constant $c_{-}\in \left( 0,c\right] $
such that $\tau _{\nu }\left( x\right) \geq c_{-}x^{2}$ for all $x>x_{0}$,
then there exists a positive constant $K$ such that for all $x>x_{0}$,%
\begin{equation*}
\mathbb{P}\left( X\geq x\right) \geq \frac{K}{x^{1+1/c_{-}}}\text{ .}
\end{equation*}%
If there exist instead $p\in \left( 0,2\right) $ and $c_{p}>0$ such that,
for all $x>x_{0}$, $\tau _{\nu }\left( x\right) \geq c_{p}x^{p}$, then there
exists a positive constant $H$ such that, for all $x>x_{0}$,%
\begin{equation*}
\mathbb{P}\left( X\geq x\right) \geq \frac{L}{x}\exp \left( -\frac{x^{2-p}}{%
\left( 2-p\right) c_{p}}\right) \text{ .}
\end{equation*}%
In the last two points, if the inequalities on $\tau _{\nu }$ in the
hypotheses are reversed, then the conclusions are also reversed, without
changing any of the constants.
\end{proposition}

\section*{Acknowledgments}

I owe thanks to Guillaume Poly for an instructive discussion about Stein's
method, to Yvik Swan for a number of pointers to the literature and to Jon
Wellner as well as Max Fathi for useful comments. I also warmly thank Michel
Bonnefont and Ald\'{e}ric Joulin for nice discussions and valuable comments
related to some (weighted) functional inequalities. Finally, I express my
gratitude to two anonymous referees for several insightful comments and
suggestions that helped to complete the study and improve the quality of the
paper.

\bibliographystyle{aalpha}
\bibliography{chern,Slope_heuristics_regression_13}

\end{document}